\documentclass[hidelinks]{inprogress}
\usepackage{listings}
\usepackage{multirow}
\usepackage{pgfplots}
\pgfplotsset{compat=1.17}
\usepackage{booktabs}
\usepackage{ifthen}
\usepackage[normalem]{ulem}
\newcommand{\rad}{r}
\newcommand{\Ball}{{B_1(\RR^d)}}
\newcommand{\optimal}[2][d]{\eta_{#1}\left({#2}\right)}
\newcommand{\optimalExpr}[2][d]{%
  \frac{%
    \varphi_{B_{1}( \RR^#1)}\left(%
      \ifthenelse{\equal{#2}{0}}{0}{#2\sqrt{\mu_{\Ball}/\lambda_{\Ball}}}%
    \right)%
  }{%
    \varphi_{B_{1} (\RR^#1)}\left(\sqrt{\mu_{\Ball}/\lambda_{\Ball}}\right)%
  }%
}
\newcommand{\optimalBessel}[2][d]{\frac{x^{1-\frac{#1}{2}}J_{\frac{#1}{2}-1}\left(#2\sqrt{\mu_{\Ball}}\right)}
{J_{\frac{#1}{2}-1}\left(\sqrt{\mu_{\Ball}}\right)}}

\begin{document}

\newauthor[fullname=Jaume de Dios Pont, 
           institution=ETH Zürich, 
           email=jaume.dediospont@math.ethz.ch]
          {Jaume}

\newauthor[fullname= Alexander Hsu, 
           institution=University of Washington, email=owlx@uw.edu]
          {Alex}

\newauthor[fullname= Mitchell A.~Taylor, 
           institution=ETH Zürich, 
           email=mitchell.taylor@math.ethz.ch]
          {Mitchell}
\title{Sharp bounds on the failure of the hot spots conjecture}

\maketitle
\abstract{
The hot spots ratio of a domain $\Omega\subset \mathbb{R}^d$ measures the degree of failure of Rauch's \emph{hot spots} conjecture on that domain. We identify the largest possible value of this ratio over all connected Lipschitz domains $\Omega\subset \mathbb{R}^d$, for any dimension $d$. As \(d\to \infty\), we show that this maximal ratio converges to \(\sqrt{e}\), which asymptotically matches the previous best known upper bound by Mariano, Panzo and Wang. For $d\ge 2$, we show that sets extremizing the hot spots ratio do not exist, and extremizing sequences must converge to a ball at a quantitative rate.  We then give a sharp bound on the measure of the set for which the first Neumann eigenfunction exceeds its maximal boundary value. From this we deduce that the \emph{hot spots} conjecture is asymptotically true ``in measure'' as \(d\to \infty\). }

\section{Introduction}

Rauch's \emph{hot spots} conjecture, first posed in his 1974 lectures at Tulane University~\cite{rauch1975five}, is a statement about the long-term behavior of solutions to the heat equation in an insulated domain. Given a bounded, connected domain \(\Omega \subset \RR^d\), let \(\psi_\Omega\) denote the first nontrivial eigenfunction of the Laplace operator in \(\Omega\) with Neumann boundary conditions.  In physical terms, \(\psi_\Omega\) represents the fundamental mode of heat dissipation in the insulated body \(\Omega\), and its behavior dominates the asymptotic temperature fluctuations for generic initial temperature distributions. Rauch's \emph{hot spots} conjecture asserted that, for sufficiently regular domains, the points where \(\psi_{\Omega}\) attains its maximal and minimal values must lie on the boundary \(\partial \Omega\). In other words, for generic initial data, the extrema of the temperature fluctuations should migrate towards \(\partial \Omega\) as \(t\to \infty\).

Initially, this conjecture was believed to hold for \emph{all} sufficiently regular domains in \(\RR^d\). Burdzy and Werner~\cite{burdzy1999counterexample} disproved this in two dimensions by constructing a counterexample with three holes (with later examples with one hole provided in~\cite{Burdzy2005TheHS}). Most attention then turned towards convex domains, where the conjecture was thought to hold in arbitrary dimensions. This was recently disproved by the first author~\cite{deDios2024convex} in all sufficiently high dimensions.

Despite these counterexamples, the conjecture has been proven for various classes of domains. In \(\RR^2\), these include all triangles~\cite{judge2020euclidean,judge2022erratum}, domains that are long and thin in different senses~\cite{banuelos1999hot,atar2004neumann,krejvcivrik2019location}, and convex sets with one axis of symmetry~\cite{jerison2000hot,pascu2002scaling}. It is widely believed that the conjecture holds for all simply connected sets in \(\mathbb R^2\).

In higher dimensions, positive results have been established for cylinder sets~\cite{kawohl1985rearrangements}, certain thin sets that are rotationally symmetric in all but one dimension~\cite{chen2019monotone}, and domains generalizing the two-dimensional class introduced by Atar and Burdzy~\cite{atar2004neumann}, such as~\cite{yang2011hot,KennedyRohleder2024}.

The existence of counterexamples in various settings raises the following fundamental question: \emph{How badly does the hot spots conjecture fail?} This question has been around since the first counterexamples were found in the late 90s, and its importance was reiterated in recent years by Steinerberger~\cite{steinerberger2023upper}, who introduced the hot spots ratio, a quantitative measure of the conjecture's failure.

\begin{definition}[\cite{steinerberger2023upper,mariano2023improved}]\label{HSR def}
Let \(S_d\) be the supremum, among all connected, bounded domains \(\Omega\subset\mathbb R^d\) with Lipschitz boundary, of the ratio
\begin{equation*}
	\frac{\max_{x\in \Omega} \psi_\Omega(x)}{\max_{x\in\partial \Omega} \psi_\Omega(x)},
\end{equation*}
where \(\psi_\Omega\) is any of the first non-constant Laplace eigenfunctions of \(\Omega\) with Neumann boundary conditions. 
\end{definition}

In this paper, we determine the exact value of \(S_d\) in all dimensions. Approximate numerical values are presented in Table~\ref{tab:values}. We further show that extremizers to this ratio do not exist for \(d>1\), and that extremizing sequences must converge to a ball at a precise quantitative rate. As \(d\to \infty\), we show that $S_d$ converges to \(\sqrt{e}\), matching the previously best known upper bound from~\cite{mariano2023improved}.
\begin{figure}[h!]
\centering

{\begin{minipage}{.5\linewidth}
    \input{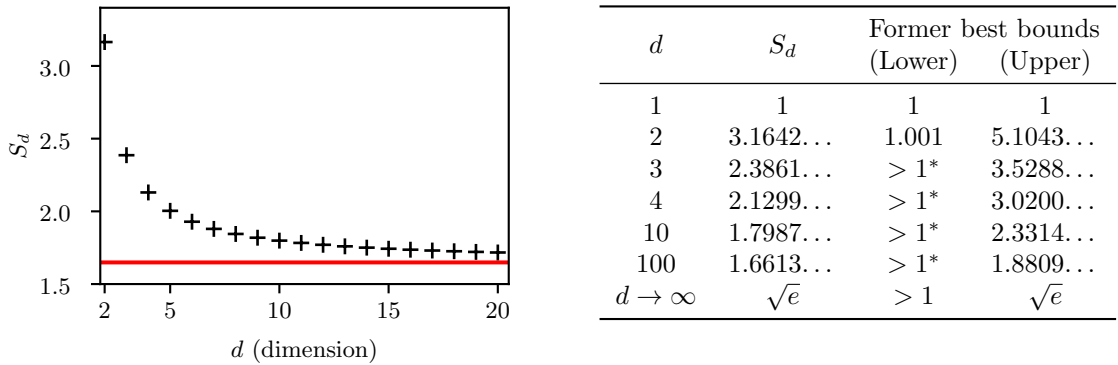}
\end{minipage}%
\begin{minipage}{.5\linewidth}
\vspace{-2em}
\begin{tabular}{ c c c c}
\toprule
\multirow{2}{*}{\(d\)} & \multirow{2}{*}{\(S_d\)} & \multicolumn{2}{c}{Former best bounds} \\
 &  & (Lower) & (Upper) \\
\midrule
1 & 1 & 1 & 1 \\
2 & 3.1642\dots & 1.001 & 5.1043\dots \\
3 & 2.3861\dots & \(>1^*\) & 3.5288\dots\\
4 & 2.1299\dots & \(>1^*\)& 3.0200\dots \\
10 & 1.7987\dots & \(>1^*\) & 2.3314\dots \\
100 & 1.6613\dots & \(>1^*\) & 1.8809\dots \\
\(d\to\infty\) & \(\sqrt{e}\) & \(>1\) & \(\sqrt{e}\) \\
\bottomrule 
\end{tabular}
\end{minipage}}

\caption{\small{(Left) Plot of $S_d$ as a function of $d$ with the the asymptotic value \(S_d\ \to \sqrt{e}\) marked by a red line. (Right) Approximate values of \(S_d\) for different dimensions \(d\) along with the previously best known bounds. The value \(>1\) means that no specific value had been computed (to the knowledge of the authors), but the value is known to be \(>1\). Results marked with \({}^\ast\) were \emph{folklore} results. The fact that \(\liminf_{d\to \infty} S_d>1\), without a specific computed value, follows from the proof of~\cite{deDios2024convex}. The  lower bound in \(d=2\) comes from numerical experiments~\cite{kleefeld2021hot}. All previously best known upper bounds follow from~\cite{mariano2023improved}. }}%
\label{tab:values}
\end{figure}

To state our main results precisely, we recall the following basic notation and terminology.

\begin{definition}[Notation for eigenfunctions] \phantom{.}
\begin{enumerate}
    \item 	We will denote by \(\varphi_{\Omega}^{(k)}\) the \(k\)-th Dirichlet eigenfunction, of eigenvalue \(\lambda_{\Omega}^{(k)}\), with \(k\) starting at \(1\). When \(k\) is omitted, it is assumed to be \(1\).

	\item We will denote by \(\psi_{\Omega}^{(k)}\) the \(k\)-th Neumann eigenfunction, of eigenvalue \(\mu_{\Omega}^{(k)}\), with \(k\) starting at \(0\) (so that \(\psi_{\Omega}^{(0)}\) is the constant function and \(\psi_{\Omega}^{(1)}\) is the first nontrivial eigenfunction).  When \(k\) is omitted, it is assumed to be \(1\).

	\item When a function (such as \(\varphi_{B_1(\mathbb R^d)}\)) is radial, we will, in an abuse of notation, write \(\varphi_{B_1(\mathbb R^d)}(x)=\varphi_{B_1(\mathbb R^d)}(|x|)\).
\end{enumerate}
\end{definition}

We will prove that sequences of domains extremizing the hot spots ratio must (in a certain sense) converge to a ball (see Figure~\ref{fig:sieve}). The reason that the ball is not a maximizer is because its first nontrivial Neumann eigenfunction is not radial. As we will see, the almost extremizing domains for the hot spots constant are slight modifications of the ball that \emph{drop} the first radial eigenvalue right below the first non-radial eigenvalue. The resulting radial eigenfunction then dictates the degree of failure of the \emph{hot spots} conjecture.
\begin{definition}
We will denote by
    \[
    \optimal{x} := \optimalExpr{x} = \optimalBessel{x}.
    \]
    This function, with domain \(\Ball\), is the unique solution to the PDE
    \[
    \begin{cases}
        -\Delta\optimal x =\mu_{\Ball} \optimal x & \text{if }|x|<1,\\
        \eta_d(x) =1 & \text{if }|x| =1.
    \end{cases}
    \]
\end{definition}
The function \(\optimal x\) is not a Neumann eigenfunction of the ball (in that it does not satisfy the Neumann boundary conditions), but it is a Laplace eigenfunction on the interior. It will arise as the locally uniform limit of a sequence of Neumann eigenfunctions of a sequence of domains converging to a ball. 

Our first main result states that the failure modes of the \emph{hot spots} conjecture are tightly controlled by the function \(\optimal{x}\). 
    \begin{theorem}%
\label{thm:goal}

    Let \(\optimal{x}\) be as above. Then for any \(d\ge 2\),
	\begin{equation*}
		S_d = \|\optimal{x}\|_{L^\infty(B_1(\RR^d))} = \optimal{0}.
	\end{equation*}
	Moreover, the value of \(S_d\) is not achieved by any set in \(\RR^d\) with Lipschitz boundary. On the other hand, any extremizing sequence of domains must converge to a ball at a quantitative rate: If \(\Omega\) has the same volume as the unit ball, then
    \[
        S_d-\frac{\max_{x\in \Omega} \psi_{\Omega}(x)}{\max_{x\in \partial \Omega} \psi_{\Omega}(x)}<\epsilon^2 \quad\implies \quad \mathcal A(\Omega) \le C_d \epsilon,
    \]
    where \(\mathcal A(\Omega)\) is the Fraenkel asymmetry of \(\Omega\).
\end{theorem}
\begin{remark}
Let \(\Omega_{res}\) be a rescaling of \(\Omega\) such that \(|\Omega_{res}|= |\Ball|\). Then the \emph{Fraenkel asymmetry}, \(\mathcal A(\Omega)\), measures how different \(\Omega_{res}\) is from the ball:
\[
\mathcal A(\Omega) := \min_{x\in \mathbb R^d} |(x+\Ball) \Delta\Omega_{res}|.
\]
\end{remark}

Since \(\optimal{x}\) can be written explicitly in terms of Bessel functions, we can compute the limit \(\lim_{d\to \infty} S_d\).
\begin{corollary}%
\label{cor:sqrte}%
    One has \(\lim_{d\to \infty} S_d = \sqrt{e}\).
\end{corollary}

\Cref{thm:goal} determines the \emph{gap} in the \emph{hot spots} conjecture in terms of the \(L^\infty\) norm. One could, alternatively, try to determine the size of the gap \emph{in measure}: How large is the set where \(\psi_\Omega\) is too large?
\begin{definition}%
    \label{def:hotspots_volume}
    For \(\alpha \in [1, S_d]\) let \(V_d(\alpha)\) be the maximum, among all bounded and connected Lipschitz domains \(\Omega\) in \(\mathbb R^d\), of the quantity \[\frac{|\{x\in \Omega \ \text{ s.t. } \ \psi^{(1)}_\Omega(x)\ge \alpha\max_{y \in \partial \Omega}\psi^{(1)}_{\Omega}(y) \}|}{|\Omega|}.\]
\end{definition}

\begin{theorem}%
    \label{thm:even_more}
   When \(\alpha \in [1, S_d]\), the function \(V_d(\alpha)\) is given  implicitly by
    \begin{equation*}
    \label{eq:volume_estimate_remark}
    V_d\left(\optimal{ { \alpha^{1/d} } }\right) = \alpha.
    \end{equation*}
\end{theorem}

\Cref{thm:even_more} implies \Cref{thm:goal} (without the stability result). The proof of both theorems is the same and it yields an analogous stability result in \Cref{thm:even_more} for any \(\alpha \in (1, S_d)\).

\Cref{cor:sqrte} shows that the \emph{hot spots} conjecture is, in an \(L^\infty\) sense, \emph{uniformly false} as \(d\to \infty\). If instead one uses a measure-theoretic sense (or, as a corollary, an \(L^p\) sense, \(0<p<\infty\)), the \emph{hot spots} conjecture becomes more and more true as the dimension increases:

\begin{corollary}\label{cor:vol_exp}
    For any \(\alpha>1\), one has \(\lim_{d\to\infty} V_d(\alpha) = 0\). Moreover, the convergence is exponentially fast.
\end{corollary}

The construction in \cite{deDios2024convex} of a convex set that violates the \emph{hot spots} conjecture is, in all but two dimensions, radially symmetric. Moreover, if in \cite[Definition 2.7]{deDios2024convex} one could choose $\psi_{\Omega,V}$ with enough freedom, one would likely be able to reach a $\sqrt{e}$ bound in \cite[Proposition 2.8]{deDios2024convex}. This suggests the following conjecture.

\begin{conjecture}
    Let \(C_d\) be the \emph{hot spots ratio} of $d$-dimensional convex sets, that is, the supremum among all convex, bounded domains \(\Omega\subset\mathbb R^d\) of the ratio
\begin{equation*}
	\frac{\max_{x\in \Omega} \psi_\Omega(x)}{\max_{x\in\partial \Omega} \psi_\Omega(x)}.
\end{equation*} Then
\begin{equation*}
    \lim_{d\to \infty} C_d = \sqrt{e}.
\end{equation*}
\end{conjecture}
\begin{remark}
    Our construction shares some overarching similarities with the counterexample to Payne's nodal line conjecture in \cite{hoffmann1997nodal} ($d=2$) or \cite{MR1836248} ($d\ge 3$), as it involves removing multiple holes from a ball, resulting in a topologically complex domain. In $d\ge 3$, a harmonic capacity argument allows one to connect the holes to the boundary of the ball, giving rise to counterexamples that are homeomorphic to a ball \cite{kennedy2013closed}. An analogous phenomenon holds for the hot spots ratio in $d\ge 3$ (but not for $d=2$, where the hot spots conjecture is still widely expected to hold for simply connected domains).
\end{remark}

\paragraph{Organization of the paper:} \Cref{sec:upper} shows the upper bound in Theorems~\ref{thm:goal} and~\ref{thm:even_more}. \Cref{sec:heuristics}  explains the heuristics for the lower bound and motivates an effective, limiting problem, which we carefully analyze in \Cref{sec:effective_problem}. \Cref{sec:neumann_sieve} constructs the sequence of sets that witness the lower bound and reduces the analysis to the effective problem from \Cref{sec:effective_problem}. The sets saturating the hot spots ratio are constructed using a limiting procedure as a spherically symmetric Neumann sieve. In \Cref{sec:asymptotics} we compute the asymptotics of $S_d$ and $V_d(\alpha)$ as $d\to \infty$, completing the proofs of Corollaries~\ref{cor:sqrte} and \ref{cor:vol_exp}.

\paragraph{LLM usage disclosure:} The sharp result arises from two independent proofs, one providing a lower bound through a sequence of examples, one providing an upper bound through a rearrangement proof.
The proof of the upper bound was found with significant assistance of a large language model (GPT o1-preview). The example witnessing the lower bound was originally found through a computational search, and then interpreted in the proof. The initial code for the lower bound example was mostly produced by GPT4o and Gemini. The proofs appearing in the manuscript were all human-written.

\paragraph{Acknowledgments:} This work was initiated when all three authors were in residence at the Mathematisches Forschungsinstitut Oberwolfach during the fall of 2024, 
participating in the Arbeitsgemeinschaft ``Quantum Signal Processing and Nonlinear Fourier Analysis''. JD was funded by the Simons Collaborations in MPS grant 563916. AH was supported by National Science Foundation grants
DMS-2208535 and DMS-2337678 with travel support to Oberwolfach provided by DMS-2230648.

\section{The upper bound}%
\label{sec:upper}

The upper bound is a refined version of the argument in~\cite{steinerberger2023upper,mariano2023improved}. We substitute estimates on stopping-time probabilities for the Brownian motion with a combination of Talenti's rearrangement inequality and the Szégo-Weinberger bound on the first eigenvalue.

\begin{definition}
    Given two nonnegative functions \(f,g\) with domains of definition \(\Omega_f, \Omega_g\) such that \(|\Omega_f|= |\Omega_g|\), we write \(f^\sharp\) to denote the symmetric decreasing rearrangement of \(f\), and \(f\le_\sharp g\) to denote that \(f^\sharp\le g^\sharp\), or, equivalently, that for all \(\alpha>0\), one has \(|\{f>\alpha\}|\le|\{g>\alpha\}|\). In particular, \(f\le g\) implies \(f\le_{\sharp} g\). 
\end{definition}

As an intermediate step in the proof, we compare \(\psi_\Omega\) with the following function.
\begin{definition}
     For a bounded domain \(\Omega \subset \mathbb R^d\) and any number \(0\le \mu<\lambda_{\Omega}\), let \(u_{\mu,\Omega}\) be the unique solution to
     \begin{equation}
	\label{eq:landscape}
		\begin{cases}
			-\Delta u_{\mu,\Omega}(x) = \mu u_{\mu,\Omega}(x) & \text{ for } x \text{ in } \Omega,\\
			u_{\mu,\Omega}(x) = 1& \text{ for } x \text{ in } \partial \Omega.
		\end{cases}
	\end{equation}
\end{definition}
Existence and uniqueness for~\eqref{eq:landscape} follow from the fact that for \(\mu<\lambda_\Omega\), the operator \(-\Delta-\mu\) is positive definite in \(H^1_{\text{Dirichlet}}(\Omega)\). With the above definitions at hand, the upper bounds in Theorems~\ref{thm:goal} and~\ref{thm:even_more} will follow from showing the chain of inequalities
\begin{equation}
\label{eq:goalchain}
\newcommand{\subeq}[2]{\underset{(\ref*{eq:goalchain}.\text{#1})}{#2}}
\frac{\psi_{\Omega}}{\max_{x\in \partial \Omega} \psi_{\Omega}(x)}
\subeq{A}{\le}
u_{\mu_{\Omega},\Omega}
\subeq{B}{\le_{\sharp}}
u_{\mu_{\Omega},\Ball} 
\subeq{C}{\le} 
u_{\mu_{\Ball},\Ball} = \eta_d.
\end{equation}
The rest of this section is devoted to the proof of these three inequalities. First, we outline the basic strategy.
\begin{enumerate}[label={(\ref*{eq:goalchain}.\Alph*)}]
    \item\label{2.A} The first inequality in \eqref{eq:goalchain} follows from a comparison principle (\Cref{prop:comparison}). Let \(\tilde \psi_\Omega:= \frac{\psi_{\Omega}}{\max_{x\in \partial \Omega}\psi_\Omega(x)}\), so that for all \(x\in \partial \Omega\) one has \(\tilde \psi_\Omega(x) \le 1 = u_{\mu_\Omega, \Omega}(x)\). Since both \(\tilde \psi_\Omega\)  and \(u_{\mu_\Omega, \Omega}\) satisfy the PDE \((\Delta + \mu_{\Omega})(\cdot) = 0\), the inequality can be extended to the interior.
    \item The second inequality is an application of Talenti's inequality (\Cref{prop:gpt_talenti}).
    \item\label{2.C}  The function \((\mu,x) \mapsto u_{\mu,\Ball}(x)\) is an analytic nonnegative function in \([0,\lambda_\Ball) \times \Ball\) (\Cref{Prop analytic}) satisfying \(\partial_\mu u_{\mu,\Ball}(x)>0\) on the interior of its domain. This implies the third inequality. 
\end{enumerate}

 Moreover, if the third step \ref{2.C} is an almost-equality at \(x=0\) (within a difference of \(\epsilon^2\)) it must be that \(\mu_\Ball-\mu_\Omega \le C_d \epsilon^2\) as well, so by the stability of the Szégo-Weinberger inequality (\cite[Theorem 4.1]{MR2899684}) the Fraenkel asymmetry of \(\Omega\) will be \(\lesssim \epsilon\).

 We begin the proof by recalling a standard comparison principle (see, e.g.,~\cite{berestycki1994principal}).
\begin{proposition}
    [Comparison principle]%
    \label{prop:comparison}
   Let \(\Omega\) be a bounded domain in \(\mathbb R^d\), with first Dirichlet eigenvalue \(\lambda_\Omega\). Let \(\mu<\lambda_\Omega\) and let \(u\in C^2(\Omega)\) be a function satisfying
    \[
    \begin{cases}
        -\Delta u \ge \mu u & \text{in } \Omega^\circ,\\
        u\ge 0 & \text{on } \partial \Omega.\\
    \end{cases}
    \]
    Then \(u\ge 0\) in \(\Omega^\circ\).
\end{proposition}
\begin{proof}
    Let \(w:=- u \cdot \1_{u\le 0}\), a positive function in \(H_{0}^1(\Omega)\) satisfying \(-\Delta w \le \mu w\) on its support. By a formal computation, we may compute that
    \[\int_{\{w\ge 0\}} |\nabla w|^2 dx = - \int_{\{w\ge 0\}} w \Delta w dx \le \mu \int_{\{w\ge 0\}} w^2  dx.\] 
    If \(w\neq 0\), by the Rayleigh quotient characterization of the first eigenvalue, we would have \(\lambda_\Omega\le \mu\), contradicting the hypothesis.

    The above \emph{integration by parts} equality requires some regularity of the boundary of the set \(\{u\le 0\}\). This can be achieved, for example, by an approximation argument considering the sets \(\{u\le -\epsilon\}\). By Sard's theorem, the sets \(\{u\le -\epsilon\}\) are generically \(C^1\), and for these generic \(\epsilon\) one does have
    \[\int_{\{w\ge \epsilon\}} |\nabla w|^2 dx = - \int_{\{w-\epsilon \ge 0\}} (w-\epsilon)_+ \Delta w  dx.\]
\end{proof}

\begin{proposition}%
\label{prop:gpt_talenti}
	For any \(\mu\le \mu_{\Ball}\) and any Lipschitz domain \(\Omega\) in \(\mathbb R^d\) such that \(|\Omega| = |{\Ball}|\), we have
    \begin{equation*}
    \label{eq:talenti_consequence_level_set}
    u_{\mu, \Omega}
    \le_\sharp
    u_{\mu, \Ball}.
    \end{equation*}
\end{proposition}
\begin{proof}
    Let \(T_{\mu,\Omega} f\) be the solution operator given by the PDE
    \[
    \begin{cases}
         - \Delta T_{\mu, \Omega}(f) = \mu f & \text{in } \Omega,\\
         T_{\mu,\Omega}(f) = 1 & \text{on } \partial \Omega.\\
    \end{cases}
    \]
    This operator is a contraction in \(L^2(\Omega)\) as long as \(\mu<\lambda_{\Omega}\) because the operator \(\frac{-1}{\mu}\Delta\) with zero Dirichlet boundary conditions has smallest eigenvalue \(\frac{\lambda_{\Omega}}{\mu}>1\). Therefore, it has as a unique fixed point \(u_{\mu,\Omega}\). By elliptic regularity, the functions \(T^{n}_{\mu,\Omega} 1\) will converge uniformly to this fixed point \(u_{\mu,\Omega}\). It suffices to show, by induction, that for all \(n\ge 0\) one has the relation \(T^{n}_{\mu, \Omega}1 \le_{\sharp} T^{n}_{\mu, \Ball}1\). 
    
    \begin{itemize}
        \item  The case \(n=0\) follows from \(1^\sharp =1\).

        \item For the induction step, we use Talenti's rearrangement inequality. Talenti's inequality implies that \(T_{\mu,\Omega} f\le_\sharp T_{\mu,\Ball} f^\sharp\). On the other hand, the maximum principle for Laplace supersolutions implies that whenever \(f^\sharp \le g^\sharp\) one has \(T_{\mu,\Ball} f^\sharp \le T_{\mu,\Ball}g^\sharp \). By combining both of these observations with the induction hypothesis, we may conclude that
    \[
    T_{\mu, \Omega}(T^n_{\mu, \Omega} 1) \le_\sharp
    T_{\mu,\Ball}( (T^n_{\mu, \Omega} 1)^\sharp)
    \le 
    T_{\mu, \Ball}( T^n_{\mu,\Ball} 1).
    \]
    \end{itemize}
    
\end{proof}

\begin{proposition}\label{Prop analytic}
    Let \(v(\mu,x):= u_{\mu, \Ball}(x)\), with domain \((0,\lambda_{\Ball})\times \Ball\). Then \(v\) is an analytic function and
    \[
    v(\mu,x) =\frac{\varphi_{\Ball}\left(|x|\sqrt{\mu/\lambda_{\Ball}}\right)}{{\varphi_{\Ball}\left(\sqrt{\mu/\lambda_{\Ball}}\right)}}.
    \]
    Moreover, \(v(\mu,x)\ge 1\) and \(\partial_\mu v(\mu,x)>0\) in \((0,\lambda_{\Ball})\times (\Ball)^\circ\).
\end{proposition}
\begin{proof}
    The function \(\varphi_{\Ball}(x)\) is a strictly positive radially decreasing function in \(\Ball\) satisfying the equation \(-\Delta \varphi_{\Ball} = \lambda_\Ball\varphi_{\Ball}\). This simultaneously proves that \(v(\mu, x)\) satisfies the PDE that defines \(u_{\mu, \Ball}(x)\) and the fact that \(v\ge 1\).

    One can take derivatives of the PDE \((-\Delta_x -\mu)v(\mu,x)=0\) to see that \(\partial_\mu v(\mu,x)\) satisfies the PDE
    \[
    \begin{cases}
        (-\Delta_x-\mu) \partial_\mu v= v \ge 1,\\
        \partial_\mu v|_{\partial \Omega} = 0.
    \end{cases}
    \]
    In particular, for some small constant \(\epsilon = \epsilon_d\) the function \(w = \partial_\mu v - \epsilon(1-|x|^2)\) satisfies the differential inequalities
    \[
    \begin{cases}
        (-\Delta_x-\mu)  w\ge 0,\\
        w |_{\partial \Omega} = 0.
    \end{cases}
    \]
    By \Cref{prop:comparison}, this shows that \(w\ge 0\), and therefore that \(\partial_\mu v \ge \epsilon (1-|x|^2)\). 
\end{proof}

\section{The lower bound construction}\label{sec:heuristics}
If one believes that the upper bounds in \Cref{sec:upper} are sharp, the stability results imply that extremizing sequences to this sharp bound must, in a certain sense, converge to a ball. The ball, however, is not a maximizer: the hot spots ratio for the ball is \(1\). 
The main obstruction in the case of the ball is that the first Neumann eigenfunction is not radial, while the upper bounds in \Cref{sec:upper} were all achieved by radially symmetric functions. In order to produce a \emph{radially symmetric} first eigenfunction, one \emph{weakly disconnects} the ball into an \emph{inner ball} and an \emph{outer annulus} by removing parts of the domain between both of them (see \Cref{fig:sieve}). 
It is then energetically favorable for the first eigenfunction to \emph{greatly change} across this weakly connected region. This favors radially symmetric eigenfunctions that are smooth on either side of the weakly connected region and that \emph{jump} across this region. 
\begin{figure}[hbt!]
    \centering
    \begin{minipage}{.3\textwidth}{\includegraphics[width=\textwidth]{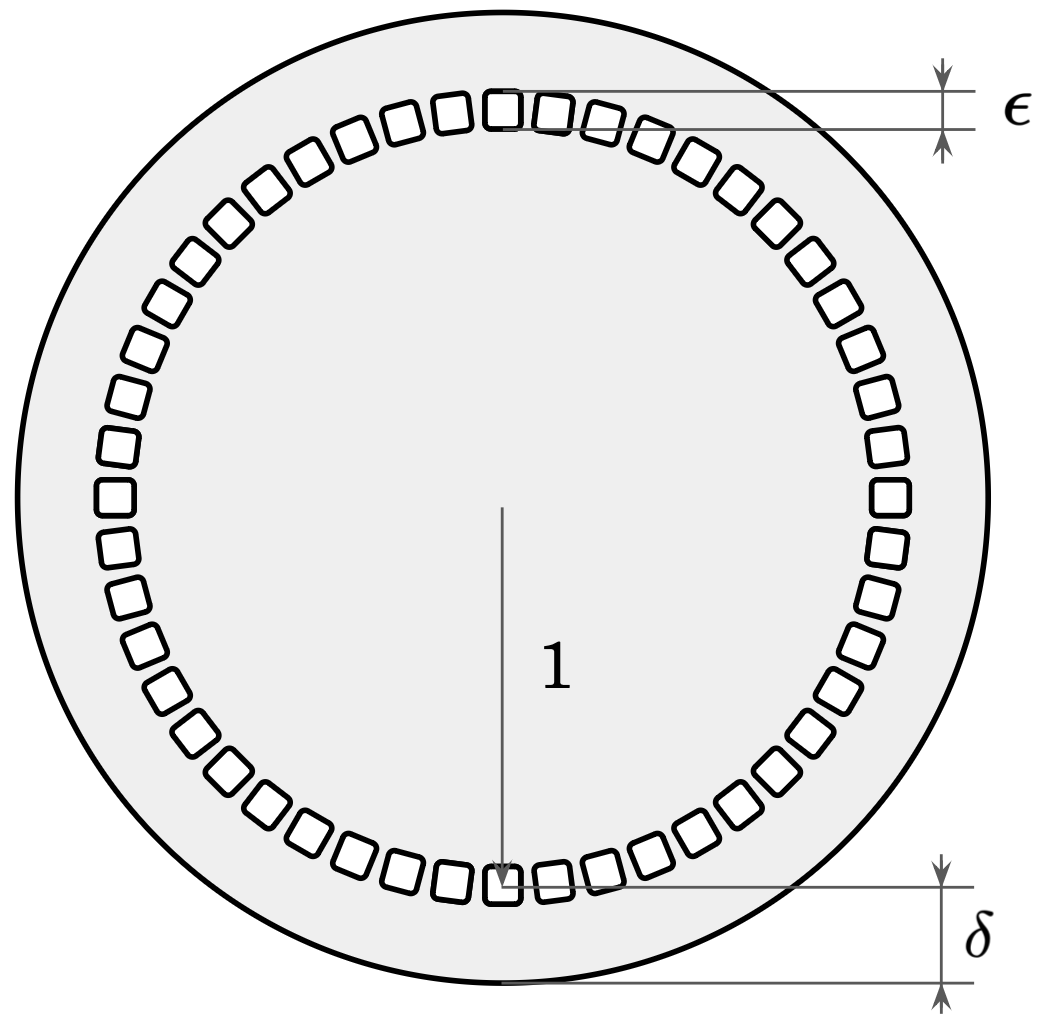}}%
    \end{minipage}
    \hspace{1em}
    \begin{minipage}{.4\textwidth}
        \includegraphics[width=0.5\linewidth]{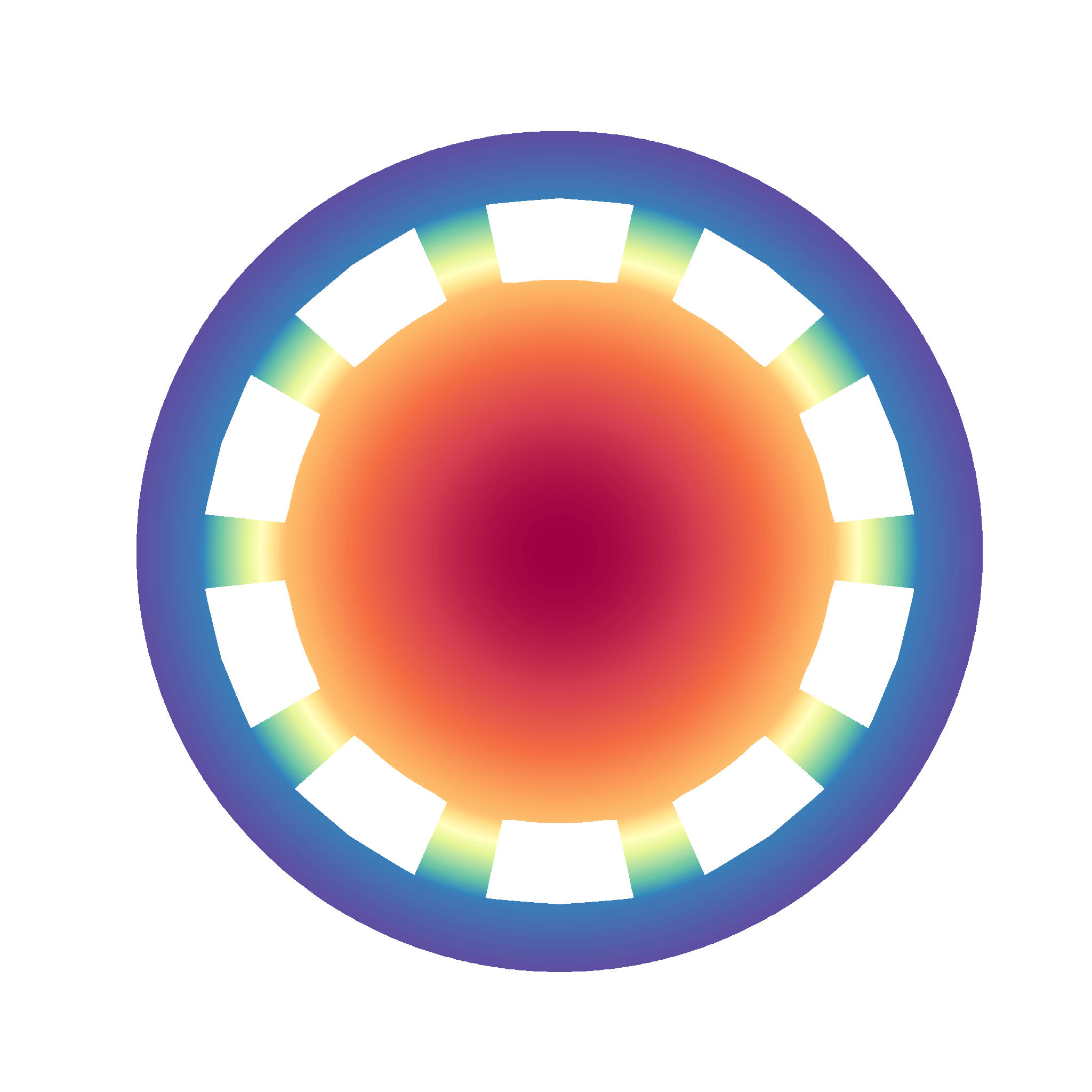}%
        \includegraphics[width=0.5\linewidth]{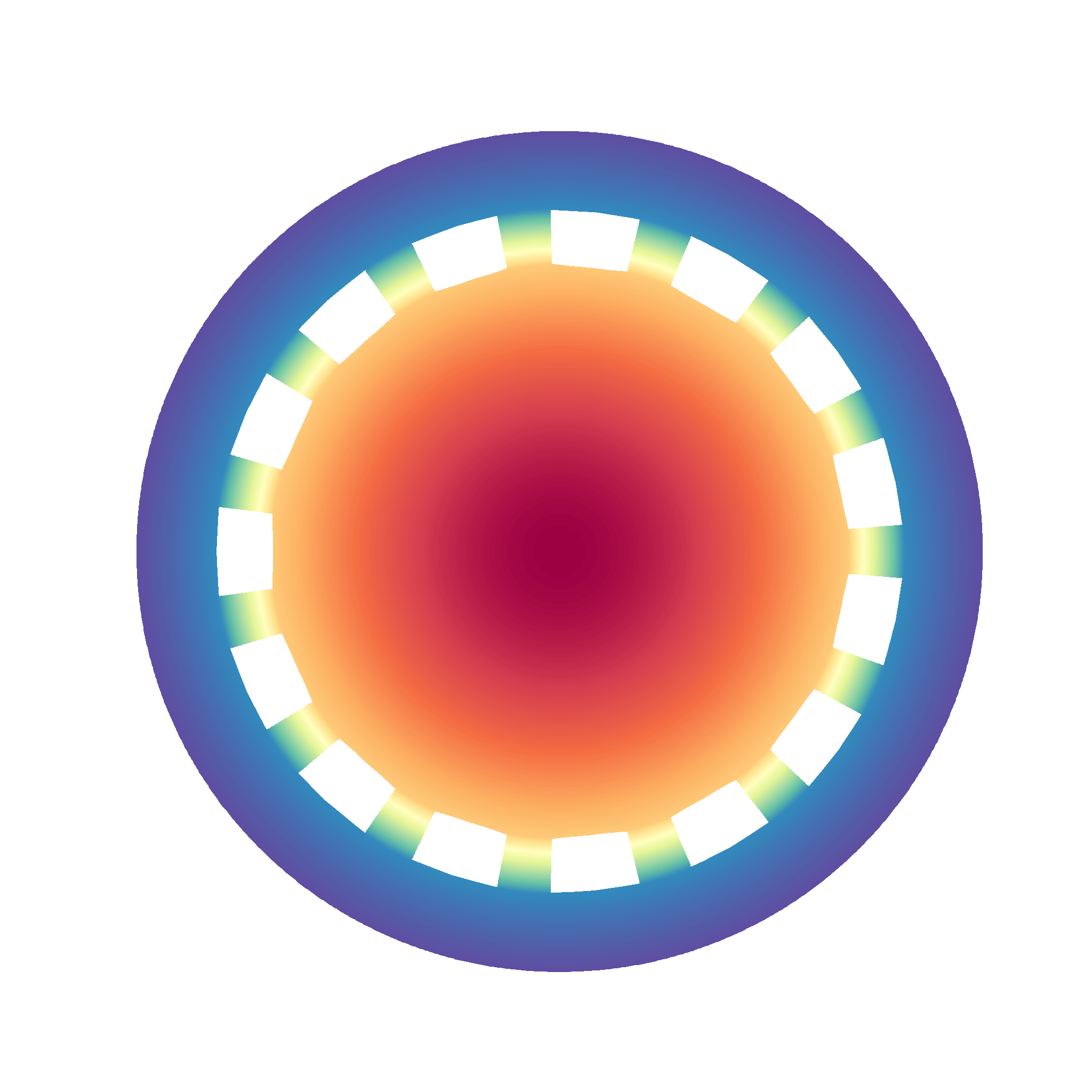}
        \\[-2.2em]    
        
        \includegraphics[width=0.5\linewidth]{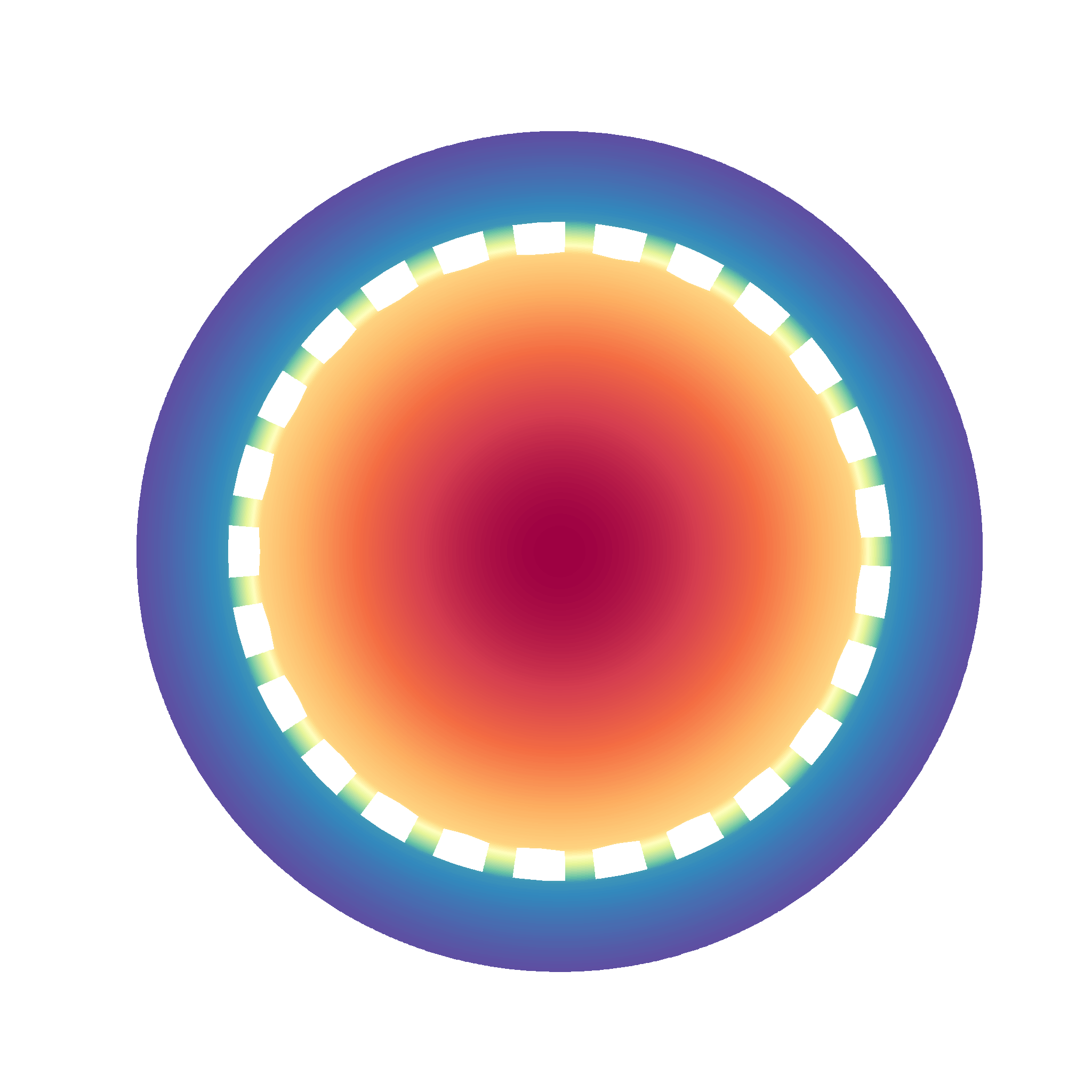}%
        \includegraphics[width=0.5\linewidth]{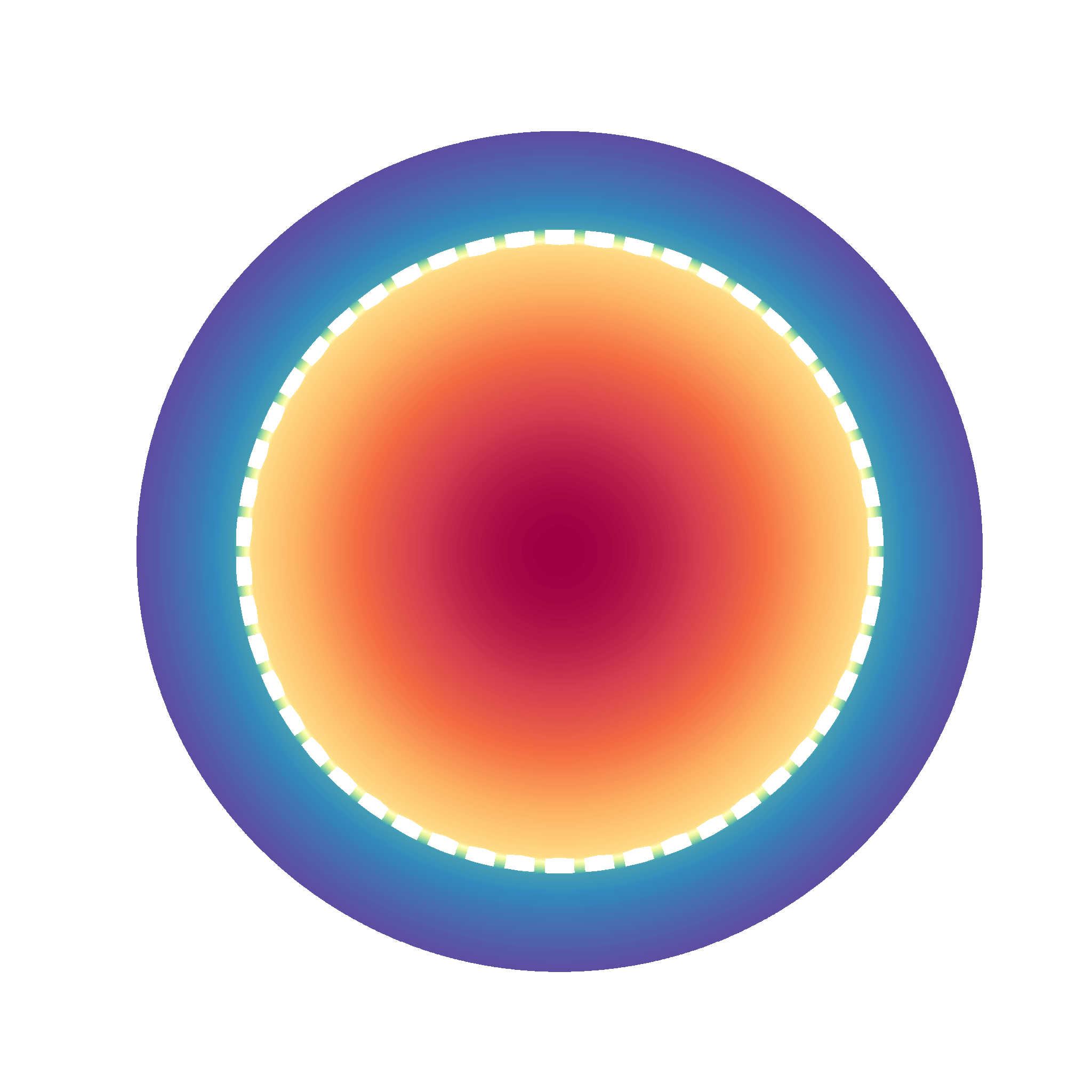}%
    \end{minipage}
    \caption{(Left) Sketch of the Neumann sieve domains. (Right) First Neumann Laplace eigenfunction for a sequence of sieve domains. In the limit when the number of holes goes to infinity one recovers a radially symmetric \emph{effective problem}. The effective problem is studied in \Cref{sec:effective_problem}, and the fact that the convergence holds is studied in \Cref{sec:neumann_sieve}.}%
    \label{fig:sieve}
\end{figure}

The process of obtaining an effective problem by adding small occlusions to a domain and sending the size of the occlusions to zero in a suitable way is known as a \emph{Neumann sieve} construction, and it has been extensively studied (see, e.g.,~\cite{marchenko1966second, damlamian1985probleme, MR916715} and \cite{khrabustovskyi2025neumann} for a modern treatment including a brief history of the problem). When the size and separation of the occlusions is sent to zero at the right rate, the limiting \emph{effective problem} has a new Robin boundary condition on a hypersurface, with a Robin parameter that one can tune. The usage of Neumann sieves to build counterexamples to the \emph{hot spots} conjecture was briefly suggested by Jerison and Nadirashvili in \cite{jerison2000hot}.

The sharp lower bound for the hot spots constant comes from carefully balancing the \emph{connectivity} between the interior and exterior regions: If the two regions are essentially disconnected, the first eigenvalue will be almost zero. Hence, the first eigenfunction will be almost constant in the inner circle and almost constant in the outer annulus, giving a hot spots ratio of almost 1. If the regions are strongly connected, the most energetically favorable eigenfunction will not be radial (as happens for the ball, where there is no disconnection at all). Balancing this connectivity parameter gives the sharp hot spots ratio.  For this reason, in our setup we will consider not only a single Neumann sieve, but a family of sieves depending on parameters \(\delta\) (the thickness of the outer annulus) and \(\beta\) (related to the connectivity strength of the sieve). The analysis proceeds as follows:
\begin{enumerate}
    \item First, we send the number of holes in \Cref{fig:sieve} to infinity and the size of the holes to zero at the right rate, in order to obtain an \emph{effective problem} depending on two parameters: the distance \(\delta\) of the sieve to the boundary, and a parameter \(\beta\) that characterizes how \emph{connected} the two sides of the Neumann sieve are to each other. This convergence is, in the appropriate sense, uniform, and one can \emph{transfer} lower bounds from this effective hot spots problem to the original hot spots problem. 
    \item The hot spots ratio of the examples arising in the effective problem can be computed explicitly. By optimizing over these parameters, (which involves sending \(\delta\) to zero and \(\beta\) to \(\mu({\Ball})\)) one achieves the desired lower bound for \(S_d\). Notice that sending \(\delta\to 0\) in these radial sieve domains makes them closer and closer to a ball, which is consistent with the \emph{equality} case in the proofs in \Cref{sec:upper} being sharp for a ball. 
\end{enumerate}

\subsection{Convergence to an effective problem}

The behavior seen in \Cref{fig:sieve} as the number of holes goes to infinity is governed by an \emph{effective bilinear form} \(D_{\beta,\delta}\) which can be written as

\begin{equation}%
\label{eq:effective_bilinear}
    D_{\beta, \delta}(f,f)=
    \int_{B_1(\mathbb R^d)} |\nabla(f)|^2 dx + \int_{B_{1+\delta}(\mathbb R^d)\setminus B_1(\mathbb R^d)} |\nabla(f)|^2 dx
    +\beta\delta\fint_{\mathbb{S}^{d-1}} |f(1^+  e)-f(1^-  e)|^2 d e,
\end{equation}
with domain \(H^1(B_{1+\delta}(\mathbb{R}^d) \setminus\SS^{d-1})\). Here, the notation \(f(1^+  e)\) (resp.~\(f(1^- e)\)) denotes the \emph{outer} (resp.~\emph{inner}) \(H^1\) trace onto the boundary \(\SS^{d-1}\). The main purpose of \Cref{sec:neumann_sieve} is to prove the following approximation result.

\begin{proposition}\label{prop:neumann_goal}
    Let \(\beta, \delta>0\) be such that \(\psi_{\beta, \delta}^{(1)}\), the first non-constant eigenfunction of \(D_{\beta,\delta}\), is radially symmetric. Then there exists a sequence of \emph{Neumann sieve} domains \(\Omega_n\) in \(\RR^d\) with first Neumann Laplace eigenfunctions \(\psi_n\) such that:
    \begin{enumerate}
        \item The domains \(\Omega_n\) converge to a ball in the Hausdorff metric.
        \item The boundaries \(\partial \Omega_n\) converge to \(\SS^{d-1} \sqcup (1+\delta) \SS^{d-1}\).
        \item There exists a choice of signs \(\sigma_n \in \{\pm 1\}\) (which we will without loss of generality assume to be $1$) such that
        \begin{enumerate}
            \item \(\sigma_n \psi_n \to \psi^{(1)}_{\beta,\delta}\) locally uniformly in \(\overline {B_{1+\delta}(\mathbb R^d)}\setminus  \SS^{d-1}\).
            \item \(\limsup_{n\to\infty} \max_{x\in \partial \Omega_n} \sigma_n\psi_n(x) = \max(\psi^{(1)}_{\beta, \delta}(1^-),\psi^{(1)}_{\beta, \delta}(1^+), \psi^{(1)}_{\beta, \delta}(1+\delta))\).
        \end{enumerate}
    \end{enumerate}
\end{proposition}

The Neumann sieve construction has been extensively studied in the literature and it is likely possible to deduce \Cref{prop:neumann_goal} from the proofs of known results. However, we were unable to find a result that implied this exact formulation, and therefore we provide a proof of \Cref{prop:neumann_goal} in \Cref{sec:neumann_sieve} for completeness.

\subsection{Optimizing over the effective problems}

\Cref{prop:neumann_goal} reduces the problem of finding lower bounds for the hot spots constant to finding lower bounds for the \emph{effective hot spots constant} associated to the operators \(D_{\beta,\delta}^{(d)}\), defined as
\[
S_d \ge \tilde S_d:= \max_{\substack{\beta>0\\ \delta>0\\ \psi^{(1)}_{\beta, \delta} \text{ is radial}}} 
\frac{
\max_{0\le r \le 1} \psi^{(1)}_{\beta,\delta }(r)}{
\max_{r\in \{1^-,1^+,1+\delta\}} \psi^{(1)}_{\beta,\delta }(r)
},
\]
where \(\psi^{(1)}_{\beta,\delta }\) is taken with the sign normalization that makes it positive at the origin. 
In \Cref{sec:effective_problem}, we will optimize over the possible values of \(\beta, \delta\) to obtain the lower bounds in Theorems~\ref{thm:goal} and~\ref{thm:even_more}. More precisely, we will prove the following proposition.

\begin{proposition}\label{prop:effective_goal}
    For any \(d\ge 2\) and all \(\beta<\mu_{\Ball}\) there is a \(\delta_0:=\delta_0(d,\beta)\) such that for all \(0<\delta<\delta_0\) the eigenfunction \(\psi^{(1)}_{\beta,\delta}\) is radial. Moreover,
    \[\lim_{\beta \to \mu_{\Ball}^{-}} \lim_{\delta \to 0} 
    \frac{\psi^{(1)}_{\beta,\delta}(r)}{\max_{r\in \{1^-,1^+,1+\delta\}} \psi^{(1)}_{\beta,\delta}(r) } = \optimal{r}.\]
\end{proposition}
\begin{figure}[hbt!]
    \centering
    \begin{minipage}[t]{0.49\linewidth}
        \centering
        \includegraphics[width=\linewidth]{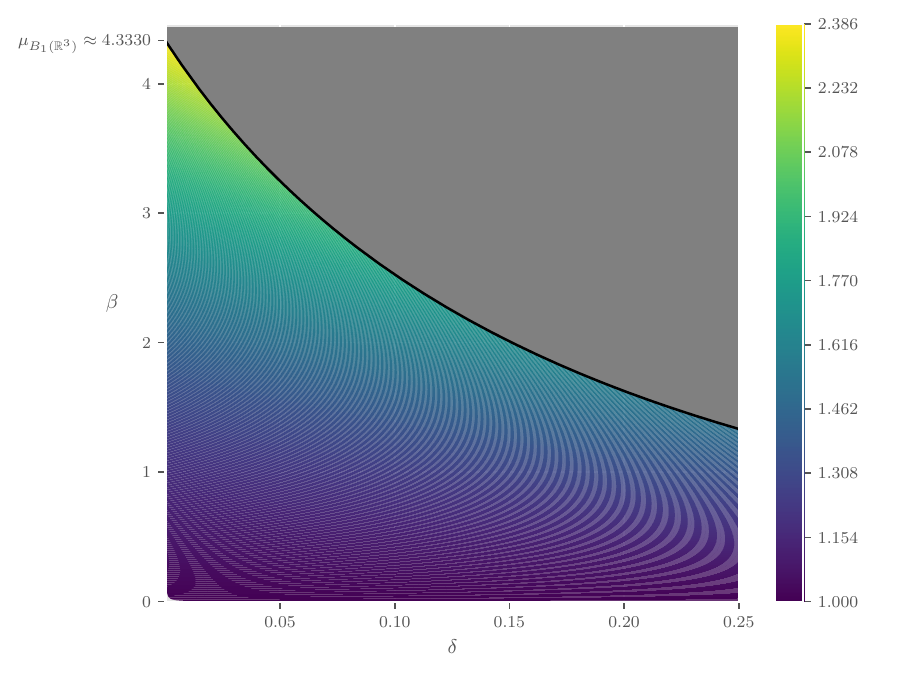}
    \end{minipage}
    \hfill
    \begin{minipage}[t]{0.49\linewidth}
        \centering
        \includegraphics[width=\linewidth]{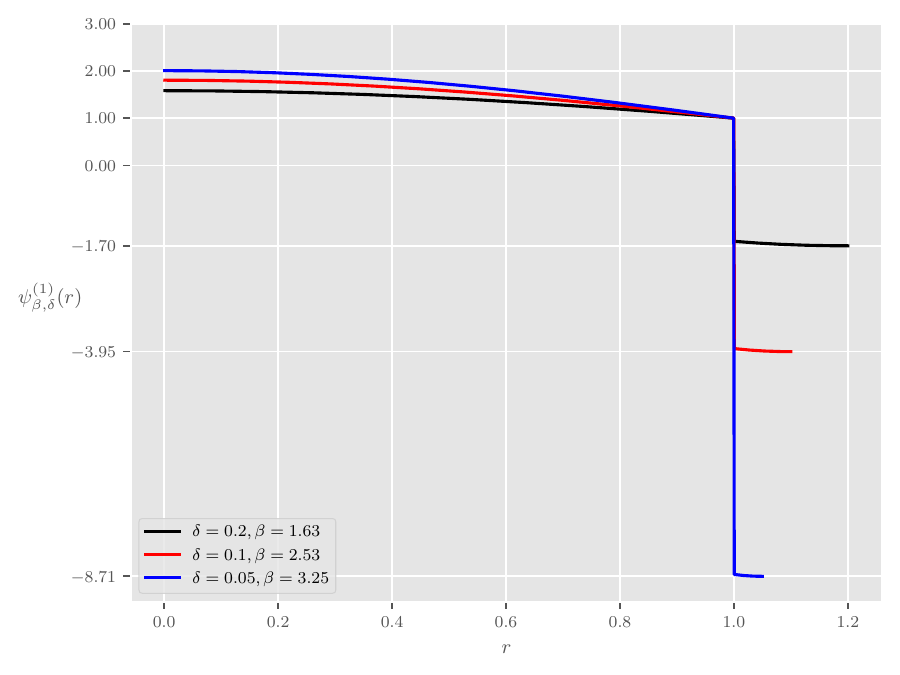}
    \end{minipage}

    \caption{(Left) Hot spots ratio in dimension 3 as a function of \(\beta\) and \(\delta\). 
    Note that the actual value of \(S_3\) is approximately \(2.3861\).
     The black curve represents the boundary of the allowed region, above which the first non-trivial Neumann eigenfunction is not radial and the effective problem yields a hot spots constant equal to 1. 
     The optimal asymptotic domain is achieved by taking \(\delta \to 0\) and \(\beta \to \mu_{B_1(\RR^3)}\).
     (Right) Examples of the radial part of the eigenfunctions 
    \(
        \frac{\psi_{\beta,\delta}^{(1)}(r)}{\psi_{\beta,\delta}^{(1)}(1^-)}
        \)
        , where \(\beta\) is taken as large as possible while keeping the first eigenfunction radial. This normalization sets the maximum on the boundary to 1, in contrast to the \(L^2\) normalization.
    As we take \(\delta \to 0\) and \(\beta \to \mu_{B_1(\RR^3)}\), 
    \(\frac{\psi_{\beta,\delta}^{(1)}(0)}{\psi_{\beta,\delta}^{(1)}(1^-)}\) will approach \(S_3\) while 
    \(
        \frac{\psi_{\beta,\delta}^{(1)}(1^+)}
        {\psi_{\beta,\delta}^{(1)}(1^-)}
        \) will tend to \(-\infty\) at a rate \(\sim-\delta^{-1}\).
        }
    \label{fig:d3hsr}
\end{figure}

\section{The effective problem}%
\label{sec:effective_problem}
Throughout this section, we fix the value of \(d\ge2\). Our aim is to understand the first nontrivial eigenfunction of the effective problem in \(L^2(B_{1+\delta}(\mathbb R^d))\), which is derived from the bilinear form \eqref{eq:effective_bilinear}.

By applying the reductions in Section~\ref{sec:neumann_sieve} and optimizing over the parameters \(\beta,\delta>0\), we will obtain the sharp lower bound for the hot spots ratio.

The bilinear form \(D_{\beta,\delta}(f,g)\) induces a self-adjoint operator and is lower bounded on its domain by the Dirichlet energy. In particular, it has compact resolvent and discrete point spectrum, with eigenfunctions \(\psi_{\beta,\delta}^{(0)},\psi_{\beta,\delta}^{(1)},\dots\), with \(\psi_{\beta,\delta}^{(0)} \equiv 1\). \Cref{prop:effective_goal} would follow if we were able to establish the following facts.

\begin{proposition}%
\label{prop:bessel_goal}
	Let \(0<\beta<\mu_{\Ball}\). Then as \(\delta\to 0\), we have:
	\begin{enumerate}
		\item\label{1} The eigenvalues \( \mu^{(1)}_{\beta, \delta}\) converge to \(\beta\).
		\item\label{2} The functions \(\psi^{(1)}_{\beta, \delta}\) are radial for all \(\delta>0\) small enough (depending on \(\beta\)).
		\item\label{3} For \(|x|<1\) (and all \(\delta>0\) small enough) the functions \(\psi^{(1)}_{\beta, \delta}\) are equal to
			\begin{equation}
    \label{eq:shape_up_to_constant}
				\psi^{(1)}_{\beta,\delta}(x) = c_{\beta,\delta} \varphi_{\Ball}\left(x \cdot \sqrt{ \mu^{(1)}_{\beta, \delta}/\lambda_{\Ball}}\right)
			\end{equation}
		for some non-zero constants \(c_{\beta,\delta}\). 
        
        \item\label{4} For \(|x|>1\) (and all \(\delta>0\) small enough) the functions \(\psi^{(1)}_{\beta, \delta}\) are negative. 
	\end{enumerate}
\end{proposition}
Notice that statement \emph{\ref{3}.}~follows immediately from \emph{\ref{2}.}~and \emph{\ref{4}.}. Since \(\psi^{(1)}_{\beta,\delta}\) is the radial Laplace eigenfunction of eigenvalue \(\mu^{(1)}_{\beta,\delta}\) for \(|x|<1\), the identity~\eqref{eq:shape_up_to_constant} follows (with \(c_{\beta,\delta}\)~potentially zero). The function \( \varphi_{\Ball}\left(x \cdot \sqrt{ \mu^{(1)}_{\beta, \delta}/\lambda_{\Ball}}\right )\) is nonnegative. Since \(\psi^{(1)}_{\beta,\delta}\) is negative for \(|x|>1\) and it must have mean zero, \(c_{\beta,\delta}\) must be positive.

\subsection{Splitting eigenfunctions by spherical harmonics}

The bilinear form \(D_{\beta,\delta}\) is invariant under rotations. In particular, there is a basis of eigenfunctions that splits (in polar coordinates) into radial functions times a spherical harmonic. The possible eigenvalues of the spherical harmonics are \(\ell(\ell+d-2)\), for integers \(l\ge 0\). For \( \beta,\delta\ge 0\), this splits the problem into a family of radial problems in one dimension, in \(L^2([0,1+\delta],  x^{d-1} dx)\), with associated forms
\begin{equation*}\label{eq:Hform}
	H_{\ell,\beta,\delta}(f,f) = H_{\beta,\delta}^0(f,f)  + \ell(\ell+d-2)\int_0^{1+\delta} |f(x)|^2 x^{d-3} dx, \hspace{5mm}l>0,
\end{equation*}
where
\begin{equation*}
	H_{\beta,\delta}^0(f,f) = \int_0^{1} |f'(x)|^2 x^{d-1} dx + \beta\delta |f(1^-)-f(1^+)|^2 + \int_1^{1+\delta} |f'(x)|^2 x^{d-1} dx.
\end{equation*}

Note that \(H^{0}_{\beta,\delta}= H_{0,\beta,\delta}\). When \(\delta = 0\), with an abuse of notation, we will define \(H_{\beta,0}^0(f,f):= \int_0^{1} |f'(x)|^2 x^{d-1} dx\) and \(H_{\ell ,\beta,0}\) accordingly. Note that \(H_{\ell ,\beta,0}\) does not depend on $\beta$.

The eigenfunctions of \(D_{\beta,\delta}\) are the products of a spherical harmonic of eigenvalue \(\ell (\ell +d - 2)\) (in the polar coordinate) and an eigenfunction of \(H_{\ell, \beta, \delta}\) (in the radial coordinate), with eigenvalue the corresponding eigenvalue of \(H_{\ell, \beta, \delta}\). We will denote by \(h^{(k)}_{\ell, \beta, \delta}\) the \(k\)-th eigenvalue of \(H_{\ell, \beta, \delta}\). The eigenvalues of \(D_{\beta,\delta}\) are then the (ordered) union of the eigenvalues of \(H_{\ell, \beta, \delta}\) over all \(\ell\), and the remaining question is which of the \(H_{\ell,\beta,\delta}\) gives rise to the first nontrivial eigenvalue of \(D_{\beta,\delta}\).

The constant function is a radial eigenfunction, showing that \(\mu^{(0)}_{\beta,\delta} = h^{(0)}_{0,\beta,\delta} = 0\). Since \(H_{\ell, \beta,\delta}\preceq H_{\ell+1,\beta,\delta}\) in the positive semidefinite ordering, we have \(h^{(0)}_{1,\beta,\delta}\le h^{(0)}_{\ell,\beta,\delta}\) for all \(\ell \ge 1\). In particular, there are only two possibilities for \(\mu^{(1)}_{\beta,\delta}\).

\begin{itemize}
    \item One possibility is that \(\mu_{\beta,\delta}^{(1)} = h_{0,\beta,\delta}^{(1)}\), i.e., the first non-zero eigenvalue of \(H_{0,\beta,\delta}(f,f) \) is the first non-zero eigenvalue of \(D_{\beta,\delta}\). In this case, the first eigenfunction of \(D_{\beta,\delta}\) is radial. 
    \item The alternative possibility is that \(\mu_{\beta,\delta}^{(1)} = h_{1,\beta,\delta}^{(0)}\). In this case, the first nontrivial eigenfunction of  \(D_{\beta,\delta}\) is not radially symmetric. Experiments show that the hot spots ratio is equal to \(1\) in this case, and therefore it is beyond our interest.
\end{itemize}

Proving \Cref{prop:bessel_goal} entails showing that we are in the first (radial) case. This will happen as long as \(h_{0, \beta, \delta}^{(1)}\le h_{1, \beta, \delta}^{(0)}\). We will first fix \(\beta<\mu_\Ball\) and send \(\delta\to 0\).

\subsection{Taking the limit \texorpdfstring{\(\delta\to 0\)}{}}

The splitting into spherical harmonics reduces \Cref{prop:bessel_goal} to the following:

\begin{proposition}%
\label{explicit_goal}
	Fix \(0<\beta<\mu_\Ball\). The following statements hold: 
    \begin{enumerate}
        \item We have 
        \begin{equation*}
	\lim_{\delta \to 0} h_{0, \beta,\delta}^{(1)} = \beta < \mu_\Ball = \lim_{\delta \to 0} h_{1, \beta, \delta}^{(0)}.
\end{equation*}
In particular, for \(\delta>0\) small enough (depending on \(\beta)\) one has \(h_{0, \beta,\delta}^{(1)}< h_{1, \beta, \delta}^{(0)}\).
        \item For \(\delta>0\) small enough, the eigenfunction corresponding to the eigenvalue \(h^{(1)}_{0,\beta,\delta }\) has constant (and opposite) signs in \([0,1)\) and in \((1, 1+\delta]\).
    \end{enumerate}
\end{proposition}
To understand the scheme of the proof of \Cref{explicit_goal},  we emphasize the following intuition coming from the simulations in \Cref{fig:d3hsr}:
\begin{enumerate}
    \item The regions $r<1$ and $r\in (1, 1+\delta)$ essentially decouple from each other.
    \item For $r \in [1,1+\delta)$, the first eigenfunction takes the value $\approx -\delta^{-1/2}$. In particular, most of the $L^2$ mass concentrates in this region.
    \item For $r \in [0,1)$, the first eigenfunction is still proportional to the radial part of a Laplace eigenfunction, with proportionality constant going to zero at a rate of $\sim \delta^{1/2}$. 
\end{enumerate}
This motivates a change of variables (in the form of an $L^2$ isometry) that \emph{focuses} on the $(1, 1+\delta)$ region.

\begin{proof}[Proof of \Cref{explicit_goal}]
We define the isometry 
\begin{equation*}
    \Phi: L^2({[0,1+\delta]}, r^{d-1}dr)\to L^2({[0,1]}, r^{d-1 }dr)\oplus L^2({[0,1]},dx)
\end{equation*}  by 
\begin{equation*}
    \Phi: f\mapsto \left(
    f|_{[0,1]},
    C_{\delta}\delta^{1/2} f(T_\delta(x))
\right),\quad
\Phi^{-1}(f,g)(r)=\begin{cases}
f(r) & r\in[0,1),\\
C_{\delta}^{-1}\delta^{-1/2}g\left(T_\delta^{-1}(r)\right) & r\in[1,1+\delta],
\end{cases}
\end{equation*}
where
\begin{equation*}
    T_\delta(x) = \left(1+d\delta C_{\delta}^2 x\right)^{1/d}
    ,\quad C_{\delta}=\sqrt{\frac{\left(1+\delta\right)^{d}-1}{d\delta}}.
\end{equation*}
We note that \(T_\delta(x) = 1 + \delta x + O(\delta^2)\) and \(C_{\delta} = 1 + o(1)\). 
The above isometry breaks \([0,1+\delta]\) into an inner function and an outer function, and induces a new bilinear form 
\(\tilde H_{\ell, \beta, \delta}(\Phi(\cdot), \Phi(\cdot)) = H_{\ell, \beta, \delta}(\cdot,\cdot)\) 
in 
\(L^2({[0,1]}, r^{d-1 }dr)\oplus L^2({[0,1]},dx)\), given by:
\begin{equation*}
\begin{aligned}
   	\tilde H_{\ell,\beta,\delta}
        ((f, g),(f, g )) 
    = 
    &\quad H_{\ell,1,0}(f,f) 
    +\beta |C_\delta\delta^{1/2} f(1) - g(0)|^2
    \\
    &+ \ell(\ell+d-2)\int_0^1 v_{\delta}(x) |g(x)|^2dx
    \\
    &+\delta^{-2}\int_0^1 w_{\delta}(x) |g'(x)|^2dx,
\end{aligned}
\end{equation*}
where
\begin{equation*}
     v_\delta(x)=\frac{T_\delta(x)^{d-3}T_\delta'(x)}{C_{\delta}^{2}\delta},\quad 
     w_\delta(x)=\frac{\delta^2}{T'_\delta(x)^{2}}.
\end{equation*}
Notably, \(v_\delta\) and \(w_\delta\) converge in \(C^\infty\) to the constant function \(1\). 
For any sequence \(\delta_n \to 0\) and any sequence \((f_n, g_n)\) such that 
\(\tilde H_{\ell, \beta, \delta_n}((f_n, g_n),(f_n, g_n))\) is uniformly bounded, 
we must have \(\int_0^1 |g_n'(x)|^2 dx \to 0\), and thus \(\|g_n-\overline g_n\|_\infty \to 0\), where \(\overline g_n\) denotes the mean of \(g_n\).

Let \((\phi^{(k)}_{\ell,\beta,\delta,}, \gamma^{(k)}_{\ell,\beta,\delta})\) be the \(k\)-th
\(L^2\)-normalized eigenfunction associated to
\(\tilde H_{\ell,\beta,\delta}\) with eigenvalue \(h^{(k)}_{\ell ,\beta, \delta}\) (in other words,
\((\phi^{(k)}_{\ell,\beta,\delta}, \gamma^{(k)}_{\ell,\beta,\delta}) = \Phi (\psi^{(k)}_{\ell,\beta,\delta})\), where \(\psi^{(k)}_{\ell,\beta,\delta}\) is the corresponding eigenfunction of the original problem). 
Note that the index \(k = 0,1,\ldots\) is for each fixed harmonic $\ell$.
The function \(\psi^{(0)}_{\ell,\beta,\delta}\) will always be constant sign, and will be constant when $\ell=0$. 

Applying Courant-Fischer and restricting the outer portion of the trial function $g$ to be constant (or even zero), 
we see that \(h^{(k)}_{\ell,\beta,\delta}\le h^{(k)}_{\ell,\beta,0} +O(1)\). 
The functions \(\gamma^{(k)}_{\ell,\beta,\delta}\) are also uniformly bounded in \(H^1\), and in particular have a convergent subsequence in 
\(H^{1-\epsilon}\) and \(L^\infty\) by Sobolev embedding, so converge (at least subsequentially) 
to a constant. 
The only interaction between \(\phi^{(k)}_{\ell,\beta,\delta}\) and \( \gamma^{(k)}_{\ell,\beta,\delta}\) is
given by 
\(2 \beta  C_\delta \delta^{1/2} \cdot \gamma^{(k)}_{\ell,\beta,\delta}(0) \cdot \phi^{(k)}_{\ell,\beta,\delta}(1)\), which goes to zero. 
In particular, as \(\delta\to 0\) the spectrum of each \(H_{\ell,\beta,\delta}\) decouples into the union of the spectrum of \(H_{\ell, \beta,0}\) (which we recall is independent of $\beta$) and the single eigenvalue \(\beta + \ell(\ell+d-2)\) coming from $\gamma$ being a constant function.

In the limit as $\delta \to 0$, the first non-zero eigenvalue of $D_{\beta,\delta}$ will come from the smallest of four possibilities:  $\beta$, $h^{(1)}_{0,\beta,0}$ (the two possibilities if the minimum is attained at $\ell = 0$), and $h^{(0)}_{1,\beta,0}$, $\beta + 1(1+d-2)$ (if the minimum is attained for $\ell=1$, in the non-radial case). 
Hence, 
we see that as long as \(\beta< h^{(0)}_{1,\beta,0} = \mu_{\Ball}\), the first eigenfunction will be in the subspace of radial functions for \(\delta>0\) small enough.

When \(\ell = 0\), 
{%
\renewcommand{\ell}{0}
\((\phi^{(0)}_{0,\beta,\delta}, \gamma^{(0)}_{0,\beta,\delta}) = c\cdot \Phi(1) = c\cdot (1,C_\delta\sqrt{\delta})\), where 
$c$ is a constant for $L^2$ normalization.
We require orthogonality for other eigenfunctions, meaning that for \(k\geq1\),
\begin{equation*}
    \int_0^1 \phi^{(k)}_{\ell,\beta,\delta}(r)r^{d-1}dr +  
    C_\delta\delta^{1/2}\int_0^1 \gamma^{(k)}_{\ell,\beta,\delta}(x)dx = 0,
\end{equation*}
which is equivalent to the mean zero condition for Neumann eigenfunctions. 
The functions \(\phi^{(k)}_{\ell,\beta,\delta}\) are the radial part of radially symmetric Laplace eigenfunctions.
Taking \(k=1\), the uniqueness of 
radial solutions to the Helmholtz equation implies that
 \(\phi^{(1)}_{\ell,\beta,\delta}\)
must be equal to 
\newcommand{\eigenvalue}{h^{(1)}_{0,\beta,\delta}}
\[ c_{\beta,\delta}\varphi_{\Ball}\left(x \cdot \sqrt{ \eigenvalue /\lambda_{\Ball}}\right).\] 
Since \(\eigenvalue < \mu_{\Ball}< \lambda_{\Ball}\) holds for small enough $\delta>0$, 
 the function  \(\phi^{(1)}_{\ell,\beta,\delta}\) will also be constant-sign (or exactly zero if \(c_{\beta,\delta}\) was zero). 
We now take the limit \(\delta \to 0\).  By the previous argument and up to subsequential limits,
\(\gamma^{(1)}_{0,\beta,\delta}\) must uniformly approach a constant \(K_{\beta}:=K_{0,\beta}^{(1)}\) (which we will assume to be $\le 0$ by the sign symmetry of eigenfunctions)
and by the explicit representation of the interior part, there exist constants $M_\beta, Q_\beta$  such that
\begin{equation*}\label{eq:mean_estimate}
    \int_0^1 \varphi_{\Ball}\left(r \cdot \sqrt{ \eigenvalue/\lambda_{\Ball}}\right)r^{d-1}dr = M_\beta+ o(1), \quad M_\beta>0,
\end{equation*}
and 
\begin{equation*}
    \int_0^1 \varphi_{\Ball}\left(r\cdot \sqrt{ \eigenvalue/\lambda_{\Ball}}\right)^2r^{d-1}dr = Q_\beta + o(1),\quad Q_\beta>0.
\end{equation*}
From orthogonality and the \(L^2\) normalization, we have
\begin{equation*}\label{eq:asy_constraints}
    c_{\beta,\delta} (M_\beta+o(1)) + \delta^{1/2} (K_\beta+o(1)) = 0, \quad
    c^2_{\beta,\delta} (Q_\beta+o(1)) + (K_\beta+o(1))^2 = 1.
\end{equation*}
Solving these equations, we obtain 
\begin{equation*}
    c_{\beta,\delta} = -\delta^{1/2}\frac{K_\beta}{M_\beta}(1 + o(1))
    \quad
    \delta\frac{K_\beta^2}{M_\beta^2} Q_\beta + K_\beta^2 = 1+o(1),
\end{equation*}
and deduce that \(K_\beta = -1\) and \(c_{\beta,\delta} >0 \) for $\delta>0$ small enough, with limits independent of the subsequence that we take. 
This establishes the second point in \Cref{explicit_goal}.
}
\end{proof}

\begin{remark}
    The effective radial problems can be analyzed more explicitly as Sturm-Liouville eigenvalue problems with an interior jump condition. The Euler-Lagrange equations associated to  
    \begin{align*}
        H_{\ell,\beta,\delta}(f,f) - \mu \int_0^{1+\delta}f(r)^2r^{d-1}dr
    \end{align*}
    are
    \begin{align*}
		  f'(1^-) &= \beta\delta \left(f(1^+) - f(1^-)\right), \\
		f'(1^-) &= f'(1^+), \\
            f'(1+\delta) &= 0, \\
	-\frac{d}{dr}\left(r^{d-1} f'(r)\right) + \ell(\ell+d-2) f(r) 
		r^{d-3} &= h^{(k)}_{\ell,\beta,\delta}r^{d-1}f(r), \quad r\in (0,1)\cup(1,1+\delta). 
    \end{align*}
    This system of equations can be transformed into a Bessel equation, and by requiring that \(f(r) = O(1)\) as \(r \to 0\), we see that, up to scaling, 
    \begin{equation*}
    f(r)=
    \begin{cases}
        r^{1-\frac{d}{2}}J_{\frac{d}{2}+\ell-1}(r\sqrt{h^{(k)}_{\ell,\beta,\delta}}) & r<1,\\
        ar^{1-\frac{d}{2}}J_{\frac{d}{2}+\ell-1}(r\sqrt{h^{(k)}_{\ell,\beta,\delta}})+br^{1-\frac{d}{2}}Y_{\frac{d}{2}+\ell-1}(r\sqrt{h^{(k)}_{\ell,\beta,\delta}}) & r>1,
    \end{cases}
    \end{equation*}
    where the coefficients \(a,b\) are chosen to match the interior jump and right boundary conditions. This alternative strategy was used in the numerical experiments to determine optimal \(\beta,\delta\) configurations and to produce \Cref{fig:d3hsr}. An analysis of the asymptotics of these problems as \(\delta \to 0\) leads to an alternative proof of \Cref{explicit_goal}. 
\end{remark}

\section{The Neumann sieve problem}
\label{sec:neumann_sieve}
\renewcommand{\rad}{{1+\delta}}

In this section, we construct a sequence of domains that saturates the upper bound for the hot spots constant from \Cref{sec:upper}. To build such domains, we use a variant of the ``Neumann sieve'' construction. Although Neumann sieve methods have been employed several times in the literature to achieve different purposes, they are usually presented in planar geometries (which are not appropriate for us) or in very general scenarios. For our purposes, a simpler, direct version of the Neumann sieve will suffice. Compared to the sieves considered in the literature, our construction is \emph{thicker}, in the sense that the length of the channels is much larger than their size (see \Cref{fig:neumann_sieve}). This greatly simplifies the homogenization computations.

\begin{figure}[hbt!]
    \begin{center}
    \begin{minipage}{0.48\textwidth}
        \includegraphics[width=\textwidth]{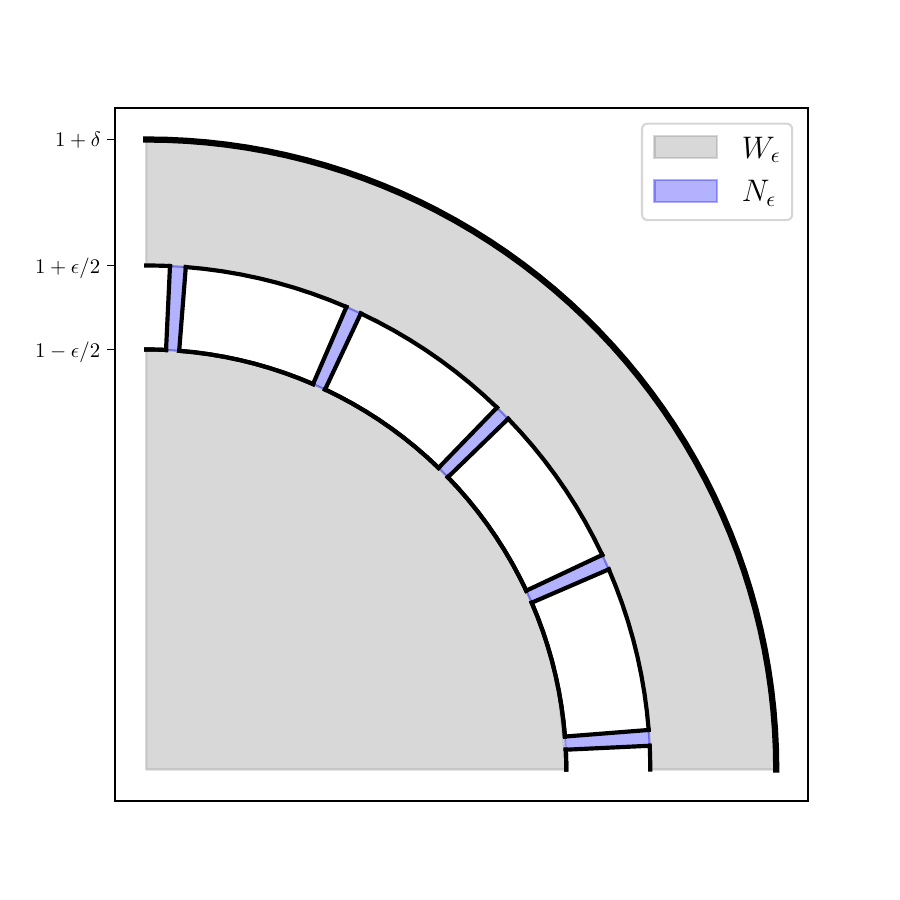}
    \end{minipage}
    \hfill
    \begin{minipage}{0.48\textwidth}
        \includegraphics[width=\textwidth]{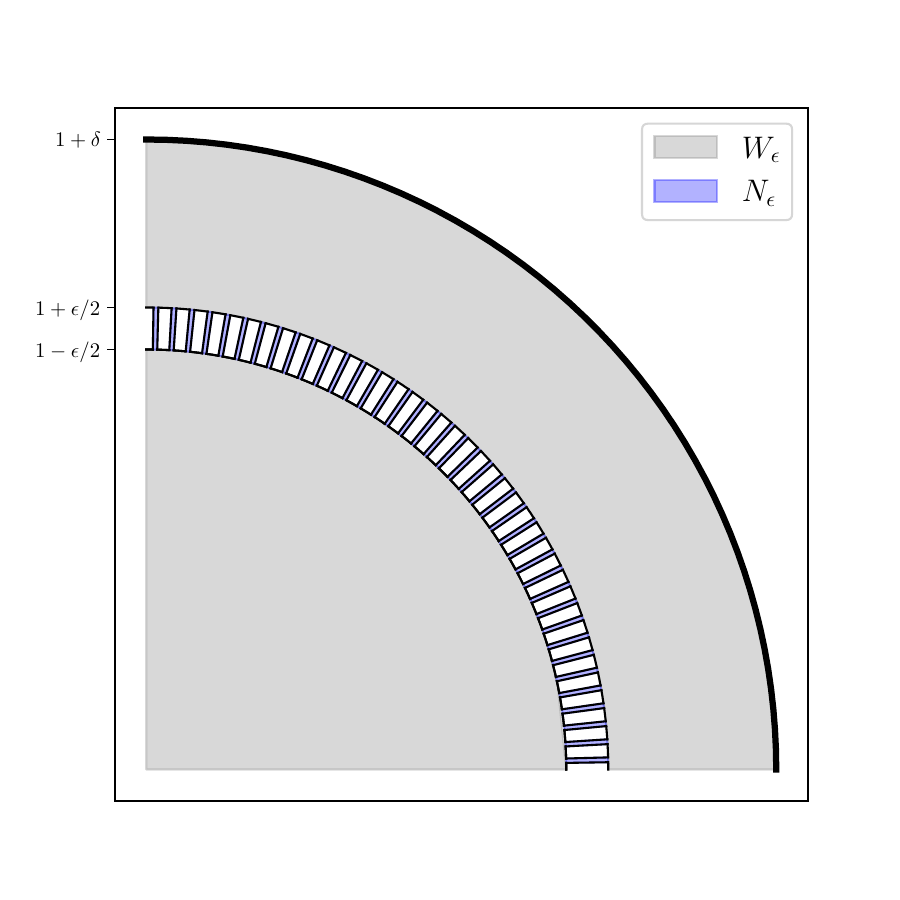}
    \end{minipage}
    \end{center}
    \caption{
        The domain \(\Omega_{\epsilon}\) consists of \(B_{\rad}:=B_{\rad}(\mathbb{R}^d)\) with the region \(\{x\in B_{\rad} \text{ s.t. } |x|\in(1-\epsilon/2,1+\epsilon/2)\}\) removed along with thin channels through that region. 
        We refer to the bulk as \(W_{\epsilon}\) and the channels as \(N_\epsilon\). 
        The width of the channels must be much smaller than their length, \(\epsilon\), being that they contain only a proportion \(\epsilon^2\alpha\) of the volume in the shell while being equidistributed on \(\SS^{d-1}\) at scales \(\sim\exp(-\frac{1}{\epsilon})\). In the case \(d=2\), this will consist of exponentially many channels that are 
        getting exponentially narrower as \(\epsilon \to 0\).
        }%
    \label{fig:neumann_sieve}%
\end{figure}

Throughout this section, we will fix \(d\ge 2\) and \(\beta, \delta>0\) such that the first nontrivial eigenfunction of \(D_{\beta,\delta}\) is unique and radial (and in particular the second eigenvalue of \(D_{\beta,\delta}\) is strictly larger than the first). We will drop the dependencies on these parameters and let \(\alpha:= \beta \delta\) be the \emph{connectivity parameter} of the Neumann sieve, i.e., the factor in front of the last summand in~\eqref{eq:effective_bilinear}.

\begin{definition}[Spherical sieve]\label{Def: Sieve description}
    Let \(0<\epsilon<\alpha^{-1/2}\). A set \(S_\epsilon \subset \SS^{d-1}\) with smooth boundary is an \emph{\(\epsilon\)-spherical sieve} if for any ball \(B_{e^{-1/\epsilon}}(x_0)\) of radius \(e^{-1/\epsilon}\) in \(\SS^{d-1}\), we have that
        \begin{equation}%
            \label{eq:sieve_error}%
           \left | \epsilon^{-2} \frac{|S_\epsilon \cap B_{e^{-1/\epsilon}}(x_0)|}{ |B_{e^{-1/\epsilon}}(x_0)|}  - {\alpha} \right|\le  e^{-1/\epsilon}.
        \end{equation}
\end{definition}
In other words, \(\epsilon\)-spherical sieves contain \(\epsilon^2\alpha\) of the volume of \(\SS^{d-1}\) and are uniformly distributed in \(\SS^{d-1}\) up to exponentially small scales \(\exp(-1/\epsilon)\) with similarly exponentially small error. 
Note that a probabilistic argument (splitting the ball into roughly \(\exp(-2/\epsilon)\) pieces, keeping each of them with probability \(\epsilon^2 \alpha\), and smoothing the result) guarantees the existence of \(\epsilon\)-spherical sieves. 
\begin{definition}[Thick sieve]\label{Def: Thick sieve description}
    A set \(\Omega_\epsilon \subset B_\rad\), with \(0<\epsilon<\delta\), is an \emph{\(\epsilon\)-thick sieve} if

    \begin{enumerate}
        \item \(\Omega_\epsilon\) is a closed domain with smooth boundary.
        \item \(B_\rad\supset \Omega_\epsilon \supset\{x\in B_\rad(\mathbb R^d) \text{ s.t. } |x|\not \in [1-\epsilon/2, 1+\epsilon/2] \}\).
        \item The domain \(B_\rad\setminus \Omega_{\epsilon}\) is \emph{graphical} over \(\SS^{d-1}\): For every \(e\in \SS^{d-1}\), the set \(\{r\ge0 \text{ s.t. } r \cdot e \in B_\rad\setminus \Omega_{\epsilon}\}\) is an interval.
        \item There exists a spherical sieve \(S_\epsilon\) approximating \(\Omega_\epsilon\), in the sense that
        \[
        \{ x\in \Omega_\epsilon \text{ s.t. } |x| \in [1-\epsilon/2+\epsilon^2, 1+\epsilon/2-\epsilon^2]\} = 
        \{e\cdot r \text{ s.t. } e\in S_\epsilon, \ r \in [1-\epsilon/2+\epsilon^2, 1+\epsilon/2-\epsilon^2]\}.
        \]
    \end{enumerate}
    The set \(\{x\in \Omega_\epsilon \text{ s.t. } |x|\in[1-\epsilon/2, 1+\epsilon/2]\}\) will be denoted by \(N_\epsilon\) (the ``necks'' of the sieve) and its complement (which contains all but \(O(\epsilon)\) of the mass in \(\Omega_\epsilon\)) will be denoted by \(W_\epsilon\).
\end{definition}
Thick sieves can be constructed from spherical sieves using an approximation argument by smoothing the set 
\[
\{x \in B_\rad \text{ s.t. } |x|\not \in [1-\epsilon/2,1+\epsilon/2] \text{ or } x/|x|\in S_{\epsilon}\}.
\]

\subsection{Estimating the eigenvalues of \texorpdfstring{\(\Omega_\epsilon\)}{}}
We now show that the eigenvalues of \(\Omega_\epsilon\) approach the eigenvalues of \(D_{\beta,\delta}\) as \(\epsilon\) goes to zero, at a rate that depends only on \(\beta\), \(\delta\) and \(d\).
\begin{proposition}\label{lem:eigenvalues_converge}
    We have
    \(\lim_{\epsilon \to 0} \mu_{\Omega_{\epsilon}}^{(k)}= \mu_{\beta,\delta}^{(k)}\),
    uniformly over the possible choices of thick Neumann sieves \(\Omega_\epsilon\), with a rate that depends only on \(\beta,\delta,d,k\).
\end{proposition}

\newcommand{\LI}{\operatorname{RI}}
\newcommand{\RE}{\operatorname{RE}}

In order to show that this is the case, we will construct two operators, \(\LI\) and \(\RE\), which approximately conjugate \(-\Delta\) in $\Omega_\epsilon$ and \(D_{\beta,\delta}\) to each other.

First, we define the linear interpolation operator \(\LI: L^2(B_{\rad}(\RR^d)) \cap H^1(B_{\rad}(\RR^d) \setminus \SS^{d-1}) \to H^1(\Omega_{\epsilon})\) as
\[\LI g(x) :=
\begin{cases}
g(x) & \text{ if } x \in W_\epsilon, \\
\frac{(1+\epsilon/2-|x|)}{\epsilon}
g\left(\left (1-\frac \epsilon 2\right)\frac x {|x|}\right) + 
\frac{(|x|-1+\epsilon/2)}{\epsilon}
g\left(\left (1+\frac \epsilon 2\right)\frac x {|x|}\right)  & \text{ if } x \in N_\epsilon.
\end{cases}
\]
In other words, the \(\LI\) operator \emph{substitutes} the value at the neck \(N_\epsilon\) by the radial linear interpolation of $g$. Analogously, we can \emph{remove} the neck and create a function with domain the whole ball of radius \(1+\delta\) by \emph{stretching} the value. We define the removal operator \(\RE : H^1(\Omega_\epsilon) \to L^2(B_{\rad}(\RR^d)) \cap H^1(B_{\rad}(\RR^d) \setminus \SS^{d-1})\) as
\[
\RE g(x)
:=
\begin{cases}
    g((1-\epsilon/2) x)& \text{ if } |x|<1,\\
    g\left(\frac x {|x|}\left(1+\frac\epsilon 2 + (|x|-1)(1- \delta^{-1}\epsilon/2 ) \right) \right) &\text{ if } |x|>1.\\
\end{cases}
\]
This operator \emph{removes} the data at the neck, mapping the ball \(B_{1-\epsilon/2}(\RR^d)\) to \(B_{1}(\RR^d)\) and the shell \(\left(B_{1+\delta}(\RR^d)-B_{1+\epsilon/2}(\RR^d)\right)\) to \(\left(B_{1+\delta}(\RR^d)-B_{1}(\RR^d)\right)\) by shifting and stretching the domains radially, combining the results, and throwing away the values of \(g\) for $|x|\in (1-\epsilon/2,1+\epsilon/2)$. 

The operators \(\LI\) and \(\RE\) are approximate isometries that approximately conjugate the Laplace operator to \(D_{\beta, \delta}\) in the following sense:

\begin{lemma}%
    \label{lem:norm_preservation}
    As \(\epsilon\to 0\), the following inequalities hold for any functions \(f\in H^1(B_{1+\delta}(\RR^d)\setminus \SS^{d-1})\) and \(g\in H^1(\Omega_{\epsilon})\):
\begin{enumerate}
    \item Mass preservation: 
        \begin{equation*}
            |\|\LI f\|_{L^2(\Omega_\epsilon)}-\|f\|_{L^2(B_{1+\delta}(\RR^d))}| \le C_{\beta,\delta,d}\epsilon\cdot (\|f\|_{L^2(B_{1+\delta}(\RR^d))}+ D_{\beta,\delta}(f,f)^{1/2}),
        \end{equation*}
        \begin{equation*}
            |\| g\|_{L^2(\Omega_\epsilon)}-\|\RE g\|_{L^2(B_{1+\delta}(\RR^d))}| \le C_{\beta,\delta,d}\epsilon\cdot (\|g\|_{L^2(\Omega_\epsilon
)}+ \|\nabla g\|_{L^2(\Omega_\epsilon)}).
        \end{equation*}
    \item Energy control:     
        \begin{equation*}
            \|\nabla \LI f\|_{L^2(\Omega_\epsilon)}^2\le D_{\beta,\delta}(f,f) + C_{\beta,\delta,d}\epsilon\cdot (\|f\|_{L^2(B_{1+\delta}(\RR^d))}^2+ D_{\beta,\delta}(f,f)),
        \end{equation*}
        \begin{equation*}
            D_{\beta,\delta}(\RE g,\RE g) \le \|\nabla g\|_{L^2(\Omega_\epsilon)}^2 + C_{\beta,\delta,d}\epsilon \cdot (\|g\|_{L^2(\Omega_\epsilon)}^2+ \|\nabla g\|_{L^2(\Omega_\epsilon)}^2).
        \end{equation*}

\end{enumerate}
\end{lemma}

\Cref{lem:norm_preservation} contains all of the geometric information about \(\Omega_\epsilon\) that is needed for the proof of \Cref{{lem:eigenvalues_converge}}. Indeed, once \Cref{lem:norm_preservation} is established, \Cref{{lem:eigenvalues_converge}} follows from a standard Courant-Fischer argument.

\begin{proof}
    [Proof of \Cref{{lem:eigenvalues_converge}}]
Let \(E_k:=\text{span}\{\LI \psi^{(i)}_{\beta,\delta}\}_{i=0}^k\) and fix \(v \in E_k\) with norm \(1\). We may express \(v = \LI \left(\sum_{i=0}^k \alpha_i \psi^{(i)}_{\beta,\delta}\right) \). Norm-preservation (together with the fact that any linear combination of the first \(k\) eigenfunctions will have bounded energy) guarantees that \(\sum_{i=0}^k \alpha_i^2 = 1+o(1)\). Energy control then guarantees that
\[
\|\nabla v\|^2_{L^2(\Omega_\epsilon)} \le \mu_{\beta,\delta}^{(k)} \sum_{i=0}^k \alpha_i^2 + o(1)(1+\mu_{\beta,\delta}^{(k)})\|v\|^2_{L^2(\Omega_\epsilon)} \le (\mu_{\beta,\delta}^{(k)}+o(1)) \|v\|_{L^2(\Omega_\epsilon)}^2.
\]
In particular, \(\mu_{\Omega_{\epsilon}}^{(k)} \le \mu_{\beta,\delta}^{(k)} + o(1)\), where the \(o(1)\) term goes to \(0\) as \(\epsilon \to 0\) at a rate depending only on \(\beta, \delta, d, k\). The symmetric argument (replacing \(\LI\) with \(\RE\)) gives \(\mu_{\beta,\delta}^{(k)} \le \mu_{\Omega_{\epsilon}}^{(k)} + o(1)\) and therefore \(\lim_{\epsilon \to 0} \mu_{\Omega_{\epsilon}}^{(k)}= \mu_{\beta,\delta}^{(k)}\), uniformly over all possible choices of thick Neumann sieves. Note that for the second argument one needs uniform upper bounds on the eigenvalues of \(\mu_{\Omega_{\epsilon}}^{(k)}\), which are provided by the first argument. 
\end{proof}

\subsection{Proof of the approximation inequalities}

The main issue in proving \Cref{lem:norm_preservation} is that the size of the spherical sieve \(S_{\epsilon}\) goes to zero as \(\epsilon\to 0\). This becomes a challenge when trying to estimate differences of the form
\begin{equation*}
    \fint_{\SS^{d-1}} g(x) dx - \fint_{S_{\epsilon}} g(x) dx.
\end{equation*}
The functions of interest arise as squares of traces of \(H^1\) functions. If $h$ is a function in $H^1$, then its trace is in $W^{1/2,2} \subset W^{1/4, 2+ 2/(2d-1)}$, and the square of its trace, by Kato-Ponce, is in $W^{1/4, 1+1/(2d-1)}\subset W^{1/4, 1}$. In particular, for $g:B_{1+\delta}(\mathbb R^d)\to \mathbb R$ we have
\begin{equation*}
    \left|\fint_{\SS^{d-1}} g^2(x) dx - \fint_{S_{\epsilon}} g^2(x) dx\right|
    \lesssim_{d,\delta} \|g\|^2_{W^{1/2,2}(\SS^{d-1})} \left \|\frac 1 {|\SS^{d-1}|} - \frac 1 {|S_\epsilon|} \1_{S_{\epsilon}}\right \|_{W^{-1/4,\infty}(\SS^{d-1})}.
\end{equation*}
\begin{lemma} [Homogenization estimate]\label{lem:homogen} We have
    \[\left \|\frac 1 {|\SS^{d-1}|} - \frac 1 {|S_\epsilon|} \1_{S_{\epsilon}}\right \|_{W^{-1/4,\infty}(\SS^{d-1})} \le C_{\beta,\delta,d} e^{- \frac{1}{5\epsilon}}.
    \]
\end{lemma}
In fact, showing convergence at a \emph{uniform} rate would be enough, and the exponential rate is not necessary. This would require allowing for balls of size \(\epsilon^{N_{d}}\) instead of \(e^{-1/\epsilon}\) as the balls for which~\eqref{eq:sieve_error} holds. This type of argument is not applicable to the more commonly studied \emph{thin} versions of the Neumann sieve, but it leads to a self-contained argument.

We will first show the energy bound estimates and derive the mass preservation estimates from those. 

\begin{proof}
    [Proof of the energy bound estimates in \Cref{lem:norm_preservation}]
    Let \(f\in H^1(B_{1+\delta}(\RR^d)\setminus\SS^{d-1})\). Then 
    \[
    \int_{\Omega_{\epsilon}} |\nabla \LI f|^2 dx =
    (1+O(\epsilon)) \left(
    \int_{B_{1+\delta}(\RR^d)\setminus\SS^{d-1}} |\nabla  f|^2 dx 
    +
    |\SS^{d-1}| \alpha \fint_{S_{\epsilon}} |f(1^- x)-f(1^+ x)|^2 dx
    \right).
    \]
    The \(W^{1/2,1}\)-norm of \(|f(1^- x)-f(1^+ x)|^2\) is bounded by the \(H^1(B_{1+\delta}(\RR^d)\setminus\SS^{d-1})\) norm of \(f\). The result then follows from the estimate in \Cref{lem:homogen}  applied to the last term, because    
    \[
    \begin{aligned}
    \bigg |\fint_{S_{\epsilon}} |f(1^- x)-f(1^+ x)|^2 dx
    - \fint_{\SS^{d-1}} |&f(1^- x)-f(1^+ x)|^2 dx \bigg |
    \\ \le \||&f(1^- x)-f(1^+ x)|^2\|_{W^{1/2,1}(\SS^{d-1})} \left \|\frac 1 {|\SS^{d-1}|} - \frac 1 {|S_\epsilon|} \1_{S_{\epsilon}}\right \|_{W^{-1/4,\infty}(\SS^{d-1})} .
        \end{aligned}
    \]
The second energy bound follows from the same approach: by the Poincaré inequality on segments, for any function  \(g\in H^1(\Omega_{\epsilon})\) we have
    \[
    \begin{aligned}
    \alpha \int_{\SS^{d-1}} |g((1+\epsilon/2) x)- g((1-\epsilon/2)x)|^2 dx 
    &= (1+O(\epsilon)) \alpha |\SS^{d-1}|   \fint_{S_{\epsilon} } |g((1+\epsilon/2) x)- g((1-\epsilon/2)x)|^2 dx 
    \\  &\le 
    \int_{N_{\epsilon}} |\partial_r g(x)|^2 dx
    \le 
    \int_{N_{\epsilon}} |\nabla g(x)|^2 dx.
    \end{aligned}
    \]
    This shows that the contributions from the neck $N_\epsilon$ asymptotically bound the contributions in the jump. 

    The norm concentration estimates are proven in essentially the same way. The key estimate states that for any function bounded in \(H^1\), there cannot be much mass in the necks, because the necks are contained in a set of size \(\lesssim \epsilon\):

    \[
    \int_{N_{\epsilon}} g(x)^2 dx \lesssim \epsilon \left( \int_{\SS^{d-1}} g((1+\epsilon/2) x)^2 + g((1-\epsilon/2) x)^2 dx + \int_{N_{\epsilon}} |\nabla g(x)|^2 dx \right).
    \]
    This shows that adding (or removing) the necks will not change the mass by much, thus proving the norm-preservation inequalities.
\end{proof}

\begin{proof}
    [Proof of \Cref{lem:homogen}]

    Let \(W(x):= \frac 1 {|\SS^{d-1}|} - \frac 1 {|S_\epsilon|} \1_{S_{\epsilon}}\). We know that \(\| W(x) \|_{L^\infty(\SS^{d-1})} \approx \epsilon^{-2}\), 
    and therefore it suffices to prove that \(\| W(x) \|_{W^{-1,\infty}(\SS^{d-1})} \lesssim e^{-\frac 4 {5\epsilon}}\). We show this by duality.
    Let \(f\in W^{1,1}(\SS^{d-1})\) of norm one. Let \(A_{s}\) be the averaging operator at scale \(s\) (averaging \(f\) on caps of radius \(s\)), which satisfies \(\|A_s f - f\|_{L^1(\SS^{d-1})}\lesssim s \|f\|_{W^{1,1}(\SS^{d-1})}\). We have

    \[
    \begin{aligned}
    \langle W, f \rangle 
    \le &
    \|W\|_{L^\infty(\SS^{d-1})}\| (A_{e^{-1/\epsilon}}-I)f\|_{L^1(\SS^{d-1})} + 
    \|A_{e^{-1/\epsilon}}W\|_{L^\infty(\SS^{d-1})}\|f-\bar{f}\|_{L^1(\SS^{d-1})}
    \\ \lesssim & 
    \epsilon^{-2}\cdot e^{-\frac 1{\epsilon}}
    +
    e^{-1/\epsilon}\cdot 1
    \lesssim e^{-\frac 4{5\epsilon}}.
    \end{aligned}.
    \]
    The bound \(\|A_{e^{-1/\epsilon}}W\|_{L^\infty(\SS^{d-1})}\lesssim 
    \epsilon^{-\frac 1{\epsilon}}\) is precisely~\eqref{eq:sieve_error}. The result then follows by interpolation.

\end{proof}

\subsection{Estimating the first eigenfunction of \texorpdfstring{\(\Omega_\epsilon\)}{} near the sieve}

By hypothesis, the first nontrivial eigenvalue of \(D_{\beta,\delta}\) is unique, with a non-zero spectral gap. In particular, we have the following lemma.

\begin{lemma}
    The functions \(\RE(\psi_{\Omega_\epsilon}^{(1)})\) must converge in \(L^2(B_{1+\delta}(\mathbb R^d))\) to \(\psi_{\beta,\delta}^{(1)}\). By elliptic regularity, this convergence holds in the \(C^\infty_{loc}\left(\overline{B_{1+\delta}(\mathbb R^d)}\setminus \SS^{d-1}\right)\) topology as well.
\end{lemma}

Near \(\SS^{d-1}\) we do not have direct access to similarly fine estimates, but we can still upper bound \(\psi_{\Omega_\epsilon}^{(1)}\) in \(N_{\epsilon}\):

\begin{lemma}
    We have \[\limsup_{\epsilon \to 0} \max_{x\in N_{\epsilon}} \psi_{\Omega_\epsilon}^{(1)}(x) \le \max_{\substack{r \in \{1^+, 1^-\}}} \psi^{(1)}_{\beta,\delta}(r).\]
\end{lemma}

\begin{proof}
    Let \(0<\epsilon<s\) be a small parameter. If \(s\) is small enough, the function 
    \[b_{s,\epsilon}(x) = ((\mu_{\beta,\delta}^{(1)}+1)(s^2 - (|x|-1)^2)+1)\cdot\max_{\substack{r \in\{1+s, 1-s\}\\ e\in \SS^{d-1}}} \psi^{(1)}_{\Omega_{\epsilon}}(r\cdot e)\]
    is a barrier function for \(\psi_{\Omega_\epsilon}^{(1)}(x)\) in \(\{x\in \Omega_{\epsilon} \text{ s.t. } |x|\in [1-s,1+s]\}\). For fixed \(s\),  uniform convergence of \(\psi^{(1)}_{\Omega_{\epsilon}}(x)\) shows that
    \[\limsup_{\epsilon \to 0} \max_{x\in N_{\epsilon}} \psi_{\Omega_\epsilon}^{(1)}(x) \le (1+(\mu_{\beta,\delta}^{(1)}+1) s^2)\max_{\substack{r \in\{1+s, 1-s\}}} \psi^{(1)}_{\beta,\delta}(r).\]
    Sending \(s\to 0\) proves the result.
\end{proof}

\section{Asymptotics as \texorpdfstring{\(d\to \infty\)}{}}\label{sec:asymptotics}
Computing our desired asymptotics as \(d\to \infty\) reduces to the following lemma:
\begin{lemma}\label{lem:gaussian_asymp}
    The functions \(\eta_d(r)\) converge uniformly for \(r\in [0,1]\) to the function
    \begin{equation*}
        \eta_{\infty}(r):= e^{(1-r^2)/2}.
    \end{equation*}
\end{lemma}

\begin{proof}
Let \(\tilde \eta_d( x) := \eta_d( x)/\|\eta_{d}\|_{\infty}\). Let \(\sigma_d\) be the uniform probability measure on the sphere of radius \(\sqrt{\mu_\Ball}\). Then \(\widehat{\sigma_d}\) is a radial eigenfunction of the Laplace operator, with the same eigenvalue as \(\tilde \eta_d( x)\), taking the same value at zero. Therefore, \(\tilde \eta_d( x)= \widehat{\sigma_d}\). Let \(\pi^{1}_* \sigma_d\) be the projection (push-forward) of \(\sigma_d\) onto the first coordinate. Then \(\tilde \eta_d(x) = \widehat{\pi^{1}_* \sigma_d}(|x|)\), and it suffices to show that \(\pi^{1}_* \sigma_d\to_{L^1} \frac{1}{{\sqrt{2\pi}}} e^{-x^2/2}\).

The eigenvalues of the ball have asymptotics \(\mu_\Ball = d+ O(1)\), and thus the measures \(\sigma_d\) are uniform on spheres of radius \(\sqrt{d}+O(d^{-1/2})\). The marginals of spheres of radius \(\sqrt{d}\) converge to standard normals, showing the convergence to a Gaussian.
\end{proof}

The fact that marginals of balls of radius \(\sqrt{d}\) converge to a Gaussian is a well-known fact in high-dimensional probability (see, e.g.~\cite{vershynin2018high}) but this can also be computed explicitly. The one-dimensional marginal of the ball is
\[
\frac{1}{Z_d} \1_{[-\sqrt{d},\sqrt{d}]}  (1- x^2/d)^{d/2} \sqrt{1 + \frac{x/d}{1-x^2/d}}
\]
which converges (up to the normalization constant) to \(e^{-x^2/2}\) as $d\to \infty$.

\Cref{cor:sqrte} follows from the fact that \(\max_{|x|<1}\eta_\infty(x)=\eta_\infty(0) = \sqrt{e}\).  \Cref{lem:gaussian_asymp} implies that, as \(d\to \infty\), \[\frac{|\{\psi^{(1)}_{\Omega}(x)>(1+\delta)\max_{y\in \partial\Omega}\psi^{(1)}_{\Omega}(y)\}|}{|\Omega|} \le (1- \delta +o(1))^{d} \lesssim e^{-d \delta(1-o(1))},\] establishing \Cref{cor:vol_exp}.

\bibliographystyle{alpha}
\bibliography{bibliography}

\appendix

\newcommand{\niceauthor}[3]{
\noindent
\begin{minipage}{.6\textwidth}%
{#1}\\{\sc #2}\\
\emph{Email address:} \texttt{#3}
\end{minipage}\\[2em]
}
\vfill

\niceauthor
{Jaume de Dios Pont} 
{Department of Mathematics, ETH Zürich}
{jaume.dediospont@math.ethz.ch}
\niceauthor
{Alexander W. Hsu}
{Department of Applied Mathematics, University of Washington}
{owlx@uw.edu}
\niceauthor
{Mitchell A.~Taylor} 
{Department of Mathematics, ETH Zürich}
{mitchell.taylor@math.ethz.ch}

\end{document}